\documentclass[12pt]{amsart}
\usepackage{txfonts}      
\usepackage{amssymb}
\usepackage{eucal}
\usepackage{amsmath}
\usepackage{amscd}
\usepackage{xcolor}
\usepackage{multicol}
\usepackage[all]{xy}           
\usepackage{graphicx}
\usepackage{color}
\usepackage{colordvi}
\usepackage{xspace}
\usepackage{tikz}
\usepackage{makecell}
\usepackage{appendix}
\usepackage{enumerate,enumitem}
\usepackage{amsthm}
\usepackage[compress]{cite}

\usepackage{ifpdf}
\ifpdf
\usepackage[colorlinks,final,backref=page,hyperindex]{hyperref}
\else
\usepackage[colorlinks,final,backref=page,hyperindex,hypertex]{hyperref}
\fi

\usepackage[active]{srcltx} 

\newtheorem{thm}{Theorem}[section]
\newtheorem{lem}[thm]{Lemma}
\newtheorem{cor}[thm]{Corollary}
\newtheorem{pro}[thm]{Proposition}
\theoremstyle{definition}
\newtheorem{defi}[thm]{Definition}

\newtheorem{rmk}[thm]{Remark}

\newcommand {\emptycomment}[1]{} 

\newcommand{\nc}{\newcommand}
\newcommand{\delete}[1]{}

\nc{\tred}[1]{\textcolor{red}{#1}}
\nc{\tblue}[1]{\textcolor{blue}{#1}}
\nc{\tgreen}[1]{\textcolor{green}{#1}}
\nc{\tpurple}[1]{\textcolor{purple}{#1}}
\nc{\tgray}[1]{\textcolor{gray}{#1}}
\nc{\torg}[1]{\textcolor{orange}{#1}}
\nc{\tmag}[1]{\textcolor{magenta}}
\nc{\btred}[1]{\textcolor{red}{\bf #1}}
\nc{\btblue}[1]{\textcolor{blue}{\bf #1}}
\nc{\btgreen}[1]{\textcolor{green}{\bf #1}}
\nc{\btpurple}[1]{\textcolor{purple}{\bf #1}}

\nc{\tforall}{\ \ \text{for all }} \nc{\hatot}{\,\widehat{\otimes}
\,} \nc{\complete}{completed\xspace} \nc{\wdhat}[1]{\widehat{#1}}

\nc{\ts}{\mathfrak{p}} \nc{\mts}{c_{(i)}\ot d_{(j)}}

\nc{\NA}{{\bf NA}} \nc{\LA}{{\bf Lie}} \nc{\CLA}{{\bf CLA}}

\nc{\cybe}{CYBE\xspace} \nc{\nybe}{NYBE\xspace}
\nc{\ccybe}{CCYBE\xspace}

\nc{\ndend}{pre-Novikov\xspace} \nc{\calb}{\mathcal{B}}
\nc{\rk}{\mathrm{r}}

\nc{\vspa}{\vspace{-.1cm}} \nc{\vspb}{\vspace{-.2cm}}
\nc{\vspc}{\vspace{-.3cm}} \nc{\vspd}{\vspace{-.4cm}}
\nc{\vspe}{\vspace{-.5cm}}

\nc{\disp}[1]{\displaystyle{#1}}
\nc{\bin}[2]{ (_{\stackrel{\scs{#1}}{\scs{#2}}})}  
\nc{\binc}[2]{ \left (\!\! \begin{array}{c} \scs{#1}\\
    \scs{#2} \end{array}\!\! \right )}  
\nc{\bincc}[2]{  \left ( {\scs{#1} \atop
    \vspace{-.5cm}\scs{#2}} \right )}  
\nc{\ot}{\otimes} \nc{\sot}{{\scriptstyle{\ot}}}
\nc{\otm}{\overline{\ot}}
\nc{\ola}[1]{\stackrel{#1}{\la}}

\nc{\scs}[1]{\scriptstyle{#1}} \nc{\mrm}[1]{{\rm #1}}

\nc{\dirlim}{\displaystyle{\lim_{\longrightarrow}}\,}
\nc{\invlim}{\displaystyle{\lim_{\longleftarrow}}\,}

\nc{\bfk}{{\bf k}} \nc{\bfone}{{\bf 1}} \nc{\rpr}{\circ}

\nc{\fraka}{{\mathfrak a}} \nc{\frakB}{{\mathfrak B}}
\nc{\frakb}{{\mathfrak b}} \nc{\frakd}{{\mathfrak d}}
\nc{\oD}{\overline{D}} \nc{\frakF}{{\mathfrak F}}
\nc{\frakg}{{\mathfrak g}} \nc{\frakm}{{\mathfrak m}}
\nc{\frakM}{{\mathfrak M}} \nc{\frakMo}{{\mathfrak M}^0}
\nc{\frakp}{{\mathfrak p}} \nc{\frakS}{{\mathfrak S}}
\nc{\frakSo}{{\mathfrak S}^0} \nc{\fraks}{{\mathfrak s}}
\nc{\os}{\overline{\fraks}} \nc{\frakT}{{\mathfrak T}}
\nc{\oT}{\overline{T}}
\nc{\frakX}{{\mathfrak X}} \nc{\frakXo}{{\mathfrak X}^0}
\nc{\frakx}{{\mathbf x}}
\nc{\frakTx}{\frakT}      
\nc{\frakTa}{\frakT^a}        
\nc{\frakTxo}{\frakTx^0}   
\nc{\caltao}{\calt^{a,0}}   
\nc{\ox}{\overline{\frakx}} \nc{\fraky}{{\mathfrak y}}
\nc{\frakz}{{\mathfrak z}} \nc{\oX}{\overline{X}}

\newcommand{\clA}{\mathcal{A}}
\newcommand{\clB}{\mathcal{B}}
\newcommand{\clE}{\mathcal{\hat{A}}}

\title[Non-abelian extensions of relative Rota-Baxter Lie algebras and Wells type exact sequences ]
{ Non-abelian extensions of relative Rota-Baxter Lie algebras and Wells type exact sequences }

\author{ Qinxiu Sun}
\address{Department of Mathematics, Zhejiang University of Science and Technology, Hangzhou, 310023} \email{qxsun@126.com}

\author{QianWen Zhu}
\address{Department of Mathematics, Zhejiang University of Science and Technology, Hangzhou, 310023}
         \email{2630583032@qq.com}


\begin{document}
\begin{abstract}

In this paper, we explore
non-abelian extensions of relative Rota-Baxter Lie algebras and 
classify the non-abelian extensions by introducing the non-abelian second cohomology group.
We also study the inducibility of a pair of automorphisms about a non-abelian extension of relative Rota-Baxter Lie algebras
and derive the Wells type exact sequences. Finally,
we investigate the inducibility problem of pairs
of derivations about an abelian extension of relative Rota-Baxter Lie algebras and give an exact
sequence of Wells type.
\end{abstract}

\maketitle

\vspace{-1.2cm}

\tableofcontents

\vspace{-1.2cm}

\allowdisplaybreaks

\section{Introduction}

Rota-Baxter operators were originally appeared in Baxter's study of the fluctuation theory in probability \cite{118} and further
was developed by Rota \cite{117}. They play an important role in the Connes-Kreimer’s algebraic approach \cite{111} to 
the renormalization in perturbative quantum field theory. Besides, Rota-Baxter operators and 
noncommutative symmetric functions and Hopf algebras have close relationship \cite{115,116}.
Recently it has found there are close connections between Rota-Baxter operators and 
double Poisson algebras \cite{113}. In the context of Lie algebra,
a Rota-Baxter operator was introduced independently in the 1980s as the operator form
of the classical Yang-Baxter equation that plays important roles in both mathematics
and mathematical physics such as integrable systems and quantum groups \cite{110,119}. B. A.
Kupershmidt developed a more general notion, an $\mathcal{O}$-operator on a Lie algebra (later
also called a generalized Rota-Baxter operator or a relative Rota-Baxter operator) in his
study of the classical Yang-Baxter equation and related integrable systems \cite{114}. 
$\mathcal{O}$-operators were known to be helpful in providing solutions of the CYBE 
 in the semidirect product Lie algebra and give rise to pre-Lie algebras
larger Lie algebra \cite{120}. Inspired by Poisson structures, Uchino
\cite{2} introduced the notion of generalized Rota-Baxter operators
on associative algebras. A relative Rota-Baxter Lie (resp. associative)
algebra is a triple consisting of a Lie (resp. associative) algebra, a representation and a relative Rota-Baxter
operator. In \cite{121,124}, the authors construct a suitable algebra which characterize relative Rota-Baxter Lie
(resp. associative) algebras as its Maurer-Cartan elements. Subsequently, representations and cohomology
of relative Rota-Baxter Lie (resp. associative, Leibniz) algebras are explicitly studied in \cite{122,123,35}.

Extensions are useful mathematical objects to understand the
underlying structures. The non-abelian extension is a relatively
general one among various extensions (e.g. central extensions,
abelian extensions, non-abelian extensions etc.). Non-abelian
extensions were first explored by Eilenberg and Maclane \cite{018},
which yielded to the low dimensional non-abelian cohomology group. This theory was subsequently extended to 
various kinds of algebras, such as Lie (super)algebras, Lie 2-algebras, Lie Yagamuti algebras, 
(relative) Rota-Baxter groups, Rota-Baxter Lie algebras and 
Rota-Baxter Leibniz algebras, see \cite{ 08, 014, 048, 021, 023, 027, 028,133} and references therein .
The cohomology and abelian extensions of relative Rota-Baxter Lie algebras were considered
in \cite{35}. But the non-abelian extension
of relative Rota-Baxter Lie algebras is still absent. This is the first motivation for writing this paper.

Another interesting study related to extensions of algebraic structures is given by the inducibility (extensibility) 
of pairs of automorphisms and derivations. When a pair of automorphisms is inducible? This problem was first
 considered by Wells \cite{037} for abstract groups and further studied in \cite{045,032}.
Since then, several authors have studied this subject further, see \cite{021,023,026,028} and references therein.
While, the extensibility problem of a pair of derivations on abelian extensions was investigated in \cite{015,130,134,135}. 
As byproducts, the Wells short exact sequences were obtained for
various kinds of algebras \cite{014, 015,021, 023,026,045,028,130}, which connected the relative automorphism groups and
the non-abelian second cohomology groups. Inspired by these results, we study inducibility of a pair of automorphisms on a non-abelian
extension of relative Rota-Baxter Lie algebras and derive the analogue of the Wells short exact sequences. 
This is another motivation for writing the present paper. 
Moreover, we explore the inducibility of pairs of derivations and develop 
the Wells exact sequences for an abelian extensions of relative Rota-Baxter Lie algebras.

The paper is organized as follows. In Section 2, we recall some basic information 
about the relative Rota-Baxter Lie algebra and its cohomology theory.
 In Section 3, we investigate non-abelian extensions and classify the non-abelian
extensions using the non-abelian second cohomology groups. In Section 4, we study
the inducibility problem of a pair of automorphisms about a non-abelian extension of relative Rota-Baxter Lie algebras.
In Section 5, we develop Wells type exact sequences in the context of non-abelian
extensions of relative Rota-Baxter Lie algebras. Finally,
we discuss the inducibility problem of pairs of derivations about an abelian extensions of relative Rota-Baxter Lie algebras.

Throughout the paper, let $k$ be a field. Unless otherwise
specified, all vector spaces and algebras are over $k$.

\setlength{\baselineskip}{1.25\baselineskip}


\section{Preliminaries on relative Rota-Baxter Lie algebras}

We begin with recalling some basic notions of relative Rota-Baxter Lie algebras following \cite{124,35}.

A Rota-Baxter Lie algebra is a Lie algebra $(A, [ \ , \
]_{A})$ with a linear map $T:
A\longrightarrow A$, satisfying
$$[T(x),T(y)]_{A}=T([T(x),y]_{A}+[x,T(y)]_{A}),~~\forall~x,y\in A.$$

\begin{defi} A relative Rota-Baxter Lie algebra is a triple $((A, [ \ ,
\ ]_A),(V,\rho),T)$, where $(A, [ \ , \
]_{A})$ is a Lie algebra, $\rho :A
\longrightarrow gl(V)$ is a representation of $A$ on the
vector space $V$ and $T:V\longrightarrow A$ is a relative
Rota-Baxter operator, i.e.,
\begin{equation*}[T(u),T(v)]_{A} = T(\rho(T(u))v- \rho(T(v))u), ~~\forall~u,
v\in V.\end{equation*}
$((A, [ \ ,
\ ]_A),(V,\rho),T)$ is called abelian if $(A, [ \ , \ ]_A)$ is an abelian Lie algebra and 
$(V,\rho)$ is a trivial representation of $(A, [ \ , \ ]_A)$, i.e.,
 $[x,y]_{A}=0$ and $\rho(x)v=0$ for all ~$x,y\in A,v\in V$.
  \end{defi}

Let $((A, [ \ , \ ]_{A}),(V,\rho),T)$ and $((A', [ \ ,
\ ]_{A'}),(V',\rho'),T')$ be two
relative Rota-Baxter Lie algebras. A homomorphism from $((A, [ \ , \ ]_{A}),(V,\rho),T)$ to $((A', [ \ ,
\ ]_{A'}),(V',\rho'),T')$ consisting of a Lie algebra homomorphism $\varphi:A\longrightarrow A'$
 and a linear map $\phi:V\longrightarrow V'$ such that
\begin{equation}\label{Irp1}T'\phi=\varphi T,\end{equation}
\begin{equation}\label{Irp2}\phi\rho(x)(v)=\rho'(\varphi(x))(\phi(v)),~~\forall~x\in A,v\in V.\end{equation}

Denote by $\mathrm{Aut}(\clA)$ the set of all automorphisms of the relative Rota-Baxter Lie algebra
$\clA=((A, [ \ , \ ]_{A}),(V,\rho),T)$. It is clear that $\mathrm{Aut}(\clA)$ is a group.

\begin{defi} A representation of a relative Rota-Baxter Lie algebra $((A
, [ \ , \ ]_{A}),(V,\rho),T)$ on a 2-term complex of
vector spaces $M\xrightarrow{S}B$ is a triple
$(\rho_{B}, \rho_{M}, \mu)$, where $\rho_{B}:A \rightarrow \mathrm{gl}(B)$ and $\rho_{M
}:A \rightarrow \mathrm{gl}(M)$ are representations of the Lie
algebra $A$ on the vector spaces $B$ and $M$
respectively, and $\mu:V\rightarrow \mathrm{Hom}(B, M)$
is a linear map such that for all $x\in
A,v\in V$,
\begin{equation}\label{Re1}\mu(\rho(x)v)=\rho_{M}(x)\mu(v)-\mu(v)\rho_{B
}(x),\end{equation}
\begin{equation}\label{Re2}\rho_{B}(T(v))S=S\rho_{M}(T(v))+S\mu(v)S.\end{equation} \end{defi}

\begin{pro} Let $\clB=(M\xrightarrow{S}B,\rho_{B}, \rho_{M}, \mu)$ be a representation of 
a relative Rota-Baxter Lie algebra $\clA=((A, [ \ , \ ]_{A}),(V,\rho),T)$. Then 
$((A\oplus B, [ \ , \ ]_{\ltimes}),(V\oplus M,\rho_{\ltimes}),T_{\ltimes})$
is a relative Rota-Baxter Lie algebra, where $ [ \ , \ ]_{\ltimes}$ is the Lie bracket defined by
\begin{equation*}[x+a , y+b]_{\ltimes}=[x,y]_{A}+\rho_{B}(x)b-\rho_{B}(y)a,~~\forall~x,y\in A,a,b\in B,
\end{equation*}
the representation $\rho_{\ltimes}:A\oplus B\longrightarrow \mathrm{gl}(V\oplus M)$ is given by
\begin{equation*}\rho_{\ltimes}(x+a)(v+m)=\rho(x)v+\rho_{M}(x)m-\mu(v)a,~~\forall~x\in A,a\in B,v\in V,m\in M,
\end{equation*}
and the relative Rota-Baxter operator $T_{\ltimes}:V\oplus M\longrightarrow A\oplus B$ is defined by
\begin{equation*}T_{\ltimes}(v+m)=T(v)+S(m),~~\forall~v\in V,m\in M.
\end{equation*}
This relative Rota-Baxter Lie algebra is called the semidirect product of $\clA=((A, [ \ , \ ]_{A}),(V,\rho),T)$ and representation
$\clB=(M\xrightarrow{S}B,\rho_{B}, \rho_{M}, \mu)$. Denote it simply by $\clA \ltimes \clB$.\end{pro}

\begin{defi} Let $\clA=((A
, [ \ , \ ]_{A}),(V,\rho),T)$ be a relative Rota-Baxter Lie algebra. A derivation $d_{\clA}=(d_A,d_{V})$ on $\clA$
 consists of linear maps $d_A\in\mathrm{Hom}(A, A)$ and $d_V\in\mathrm{Hom}(V, V)$
such that $d_A$ is a derivation on $A$ and the following equalities are satisfied:
\begin{equation}\label{D1}Td_{V}=d_{A}T,
\end{equation}
\begin{equation}\label{D2}d_{V}(\rho(x)v)=\rho(x)d_{V}(v)+\rho(d_{A}(x))v,~~\forall~x\in A,~
v\in V.
\end{equation}
\end{defi}

Denote the set of all derivations on $\clA=((A
, [ \ , \ ]_{A}),(V,\rho),T)$ by $\mathrm{Der}(\clA)$. It is easy to check that $\mathrm{Der}(\clA)$ is a Lie algebra.

Next, let us recall the cohomology theory of relative Rota-Baxter Lie
algebras investigated in \cite{35}.
 Let $(M\xrightarrow{S}B,\rho_{B}, \rho_{M}, \mu)=\clB$ be a representation
of a relative Rota-Baxter Lie algebra $((A, [ \ , \
]_{A}),(V,\rho),T)=\clA$. In view of Lemma 3.1 \cite{125}, $(A,\rho_T)$ is a representation
 of the Lie algebra $(V,[ \ , \  ]_{T})$ with $\rho_T(u)(x)=[Tu,x]_{A}+T\rho(x)(u),~
 [u,v]_T=\rho(Tu)v-\rho(Tv)u,~\forall~x\in A,u,v\in V$.

Define the $n$-cochains group $C^{n}(\clA,\clB)$ as follows:
\begin{equation}\label{DLY}
	\mathcal{C}^{n}(\clA,\clB)=\left\{
	\begin{aligned}
		&\mathrm{Hom}(A,B)\oplus \mathrm{Hom}(V,M),&n=1,\\
		&\Big(\mathrm{Hom}(\wedge^{n}A,B)\oplus
\mathrm{Hom}(\wedge^{n-1}A\otimes V,M)\Big) \oplus
\mathrm{Hom}(\wedge^{n-1}V,B),&n\geq 2.
	\end{aligned}
	\right.
\end{equation}
 The coboundary
operator $$\mathcal{D_{R}}:\mathcal{C}^{n}(\clA,\clB)\longrightarrow \mathcal{C}^{n+1}(\clA,\clB)$$
 is given by
\begin{equation}\label{eq3.1}\mathcal{D_{R}}(f,\theta)=(\delta f,\partial \theta+h_{T}(f))
=((\delta f)_{B},(\delta f)_{M},\partial
\theta+h_{T}(f)),\end{equation} for any pair $f=(f_{B},f_M)\in \mathrm{Hom}(\wedge^{n}A,B)\oplus
\mathrm{Hom}(\wedge^{n-1}A\otimes V,W)$ and $\theta\in
\mathrm{Hom}(\wedge^{n-1}V,B)$. More precisely, the maps
\begin{equation*}\delta:\mathrm{Hom}(\wedge^{n}A,B)\oplus
\mathrm{Hom}(\wedge^{n-1}A\otimes V,M)\longrightarrow
\mathrm{Hom}(\wedge^{n+1}A,B)\oplus
\mathrm{Hom}(\wedge^{n}A\otimes V,M),\end{equation*}
\begin{equation*}\partial:\mathrm{Hom}(\wedge^{n-1}V,B)\longrightarrow
\mathrm{Hom}(\wedge^{n}V,B),\end{equation*}
\begin{equation*}h_{T}(f):\mathrm{Hom}(\wedge^{n}A,B)\oplus
\mathrm{Hom}(\wedge^{n-1}A\otimes V,M)\longrightarrow
\mathrm{Hom}(\wedge^{n}V,B),\end{equation*} are defined respectively by
\begin{align}\label{eq3.2}&(\delta f)_B(x_1,\cdot\cdot\cdot,x_{n+1})\nonumber
\\=&\sum_{i=1}^{n+1}(-1)^{i+1}\rho_{B}(x_i)f_{B}(x_1,\cdot\cdot\cdot,\hat{x}_{i},\cdot\cdot\cdot,x_{n+1})
\nonumber
\\&+ \sum_{1\leq i< j\leq n+1}(-1)^{i+j}f_{B
}([x_i,x_j]_{A},x_1,\cdot\cdot\cdot,\hat{x}_{i},\cdot\cdot\cdot,\hat{x}_{j},\cdot\cdot\cdot,x_{n+1}),
\end{align}
\begin{align}\label{eq3.3}&(\delta f)_{M}(x_1,\cdot\cdot\cdot,x_{n},v)\nonumber
\\=&\sum_{1\leq i< j\leq n}(-1)^{i+j}f_{M
}([x_i,x_j]_{A},x_1,\cdot\cdot\cdot,\hat{x}_{i},\cdot\cdot\cdot,\hat{x}_{j},\cdot\cdot\cdot,x_{n},v)\nonumber
\\&-(-1)^{n-1}\mu(v)f_{B}(x_1,\cdot\cdot\cdot,\hat{x}_{i},\cdot\cdot\cdot,x_{n}) \nonumber
\\&+
\sum_{i=1}^{n}(-1)^{i+1}(\rho_{M}(x_i)f_{M
}(x_1,\cdot\cdot\cdot,\hat{x}_{i},\cdot\cdot\cdot,x_{n},v) -f_{M
}(x_1,\cdot\cdot\cdot,\hat{x}_{i},\cdot\cdot\cdot,x_{n},\rho(x_i)v),
\end{align}
\begin{align}\label{eq3.4}&(\partial \theta)(v_1,\cdot\cdot\cdot,v_{n})\nonumber
\\=&\sum_{i=1}^{n}(-1)^{i+1}\rho_{B}(Tv_i)\theta(v_1,\cdot\cdot\cdot,\hat{v}_{i},\cdot\cdot\cdot,v_{n})
-\sum_{i=1}^{n}(-1)^{i+1}
S(\mu(v_i)\theta(v_1,\cdot\cdot\cdot,\hat{v}_{i},\cdot\cdot\cdot,v_{n}))
\nonumber
\\&+ \sum_{1\leq i< j\leq
n}(-1)^{i+j}\theta(\rho(Tv_i)v_j-\rho(Tv_j)v_i,v_1,\cdot\cdot\cdot,\hat{v}_{i},\cdot\cdot\cdot,\hat{v}_{j},\cdot\cdot\cdot,v_{n}),
\end{align}
\begin{align}\label{eq3.5}&(h_{T}f)(v_1,\cdot\cdot\cdot,v_{n})\nonumber
\\=&(-1)^{n}f_{B}(Tv_1,\cdot\cdot\cdot,Tv_n)
+\sum_{i=1}^{n}(-1)^{i+1}
Sf_{M}(Tv_1,\cdot\cdot\cdot,Tv_{i-1},Tv_{i+1},\cdot\cdot\cdot,Tv_{n},v_i),
\end{align}
for any $x_{1},\cdot\cdot\cdot,x_{n}\in A$ and
$v_{1},\cdot\cdot\cdot,v_{n},v\in V$.

Associated to the cochain complex
  $(\mathcal{C}^{*}(\clA,\clB), \mathcal{D_{R}})$, the set of all $n$-cocycles and
$n$-coboundaries are denoted respectively by $\mathcal{Z}^{n}(\clA,\clB)$ and $\mathcal{B}^{n}(\clA,\clB)$. 
Denote the quotient by
\begin{equation}\label{DLY}
	\mathcal{H}^{n}(\clA,\clB)=\left\{
	\begin{aligned}
		&\mathcal{Z}^{1}(\clA,\clB ),&n=1,\\
		&\mathcal{Z}^{n}(\clA,\clB )/\mathcal{B}^{n}(\clA,\clB),&n\geq 2,
	\end{aligned}
	\right.
\end{equation}
which is called the $n$-cohomology group of
 $\clA$ with
 coefficients in the representation $\clB$.

We interpret the 1-cocycyle and 2-cocycle.

For all $\varphi=(\varphi_B,\varphi_M)\in \mathcal{C}^{1}(\clA,\clB)$, write
$\mathcal{D_{R}}(\varphi_B,\varphi_M)=((\delta\varphi)_B,(\delta\varphi)_M,h_{T}(\varphi))$.
 By direct calculation, for all $x,y\in A,v\in V$, we get
\begin{equation}\label{Coy1}(\delta \varphi)_{B}(x,y)=\rho_{B}(x)\varphi_{B}(y)-\rho_{B}(y)\varphi_{B}(x)-\varphi_{B}([x,y]_A)
,\end{equation}
\begin{equation}\label{Coy2}(\delta \varphi)_{M}(x,v)=-\mu(v)\varphi_{B}(x)+\rho_{M}(x)\varphi_{M}(v)-\varphi_{M}(\rho(x)v)
,\end{equation}
\begin{equation}\label{Coy3}(h_{T} \varphi)(v)=-\varphi_{B}(Tv)
+S\varphi_{M}(v).\end{equation}
Thus,
\begin{align}
		\mathcal{Z}^{1}(\clA,\clB )=&\left\{\varphi=(\varphi_B,\varphi_M)\in \mathcal{C}^{1}(\clA,\clB)
\left|\begin{aligned}&
   \rho_{B}(x)\varphi_{B}(y)-\rho_{B}(y)\varphi_{B}(x)=\varphi_{B}([x,y]_A),
   \\&\mu(v)\varphi_{B}(x)+\varphi_{M}(\rho(x)v)=\rho_{M}(x)\varphi_{M}(v),
     \\&\varphi_{B}(Tv)=S\varphi_{M}(v),~\forall~x,y\in A,v\in V.
     \end{aligned}\right.\right\}.\label{Cy1}
	\end{align}
\begin{align}
		\mathcal{B}^{2}(\clA,\clB )=&\left\{((\delta\varphi)_B,(\delta\varphi)_M,h_{T}\varphi)\in \mathcal{C}^{2}(\clA,\clB)
|\varphi=(\varphi_B,\varphi_M)\in \mathcal{C}^{1}(\clA,\clB)
    \right\}.\label{Cb1}
	\end{align}
 Write $\mathcal{D_{R}}(f,\theta)=((\delta f)_B,(\delta f)_M,\partial\theta +h_{T}(f))$
for all $(f,\theta)\in \mathcal{C}^{2}(\clA,\clB)$ with $f=(f_B,f_M)$.
By direct computations, for all $x,y,z\in A,v,v_1,v_2\in V$, we obtain
 \begin{align*}(\delta f)_{B}(x,y,z)=&\rho_{B}(x)f_{B}(y,z)+\rho_{B}(z)f_{B}(x,y)-\rho_{B}(y)f_{B}(x,z)
  \\&-f_{B}([x,y]_A,z)-f_{B}([y,z]_A,x)+f_{B}([x,z]_A,y),\end{align*}
 \begin{align*}(\delta f)_{M}(x,y,v)=&-f_{M}([x,y]_A,v)+\mu(v)f_{B}(x,y
)+\rho_{M}(x)f_{M}(y,v)\\&-\rho_{M}(y)f_{M}(x,v)-f_{M}(x,\rho_{V}(y)v)+f_{M}(y,\rho_{V}(x)v),\end{align*}
\begin{align*}(\partial \theta+h_{T} f)(v_1,v_2)
=&\rho_{B}(Tv_1)\theta(v_2)-\rho_{B}(Tv_2)\theta(v_1)-S(\mu(v_1)\theta(v_2)-\mu(v_2)\theta(v_1))
-\theta(\rho_{B}(Tv_1)v_2
\\&-\rho_{B}(Tv_2)v_1)+f_{B}(Tv_1,Tv_2)-S(f_{M}(Tv_1,v_2)-f_{M}(Tv_2,v_1)).\end{align*}
So, \begin{align}
		\mathcal{Z}^{2}(\clA,\clB )=&\left\{(f_B,f_M,\theta)\in \mathcal{C}^{2}(\clA,\clB)
\left|\begin{aligned}&
 (\delta f)_{B}(x,y,z)=0,~(\delta f)_{M}(x,y,v)=0,\\&(\partial \theta+h_{T} f)(v_1,v_2)=0,~\forall~x,y,z\in A,v,v_1,v_2\in V.
     \end{aligned}\right.\right\}.\label{Cy2}
	\end{align}


\section{Non-abelian
extensions and non-abelian 2-cocycles of relative Rota-Baxter Lie algebras}

In this section, we are devoted to considering non-abelian
extensions and non-abelian 2-cocycles of relative Rota-Baxter Lie algebras. More precisely,
we define the non-abelian second cohomology group and show that
the non-abelian extensions are classified by the non-abelian second cohomology groups.

\begin{defi} Let $((A, [ \ , \ ]_{A}),(V,\rho),T)=\clA$ and $((B, [ \ , \ ]_{B}),(M,\nu_M),S)=\clB$ be two
relative Rota-Baxter Lie algebras.
A non-abelian 2-cocycle on $\clA$ with values in
 $\clB$ is a sextuple $(\omega,\varpi,\chi,\mu,\rho_B,\rho_M)$
of maps such that $\omega:A\otimes A\longrightarrow B,~\varpi:A\otimes V\longrightarrow M$ are bilinear
and $\rho_B:A\longrightarrow \mathfrak{gl}(B),\rho_M:A\longrightarrow \mathfrak{gl}(M)
~\mu:V\longrightarrow \mathrm{Hom}(B,M),~\chi:V\longrightarrow B$ are linear,
and the following equations are satisfied for all
$ x, y, z\in A, ~a\in  B, ~v,v_1,v_2\in V, ~m\in M$,
\begin{equation}\label{L1}
\omega(x,y)+\omega(y,x)=0,
\end{equation}
\begin{equation}\label{L2}\rho_{B}(x)\rho_{B}(y)a-\rho_{B}(y)\rho_{B}(x)a-\rho_{B}([x,y]_A)a=[\omega(x,y),a]_B,\end{equation}
\begin{equation}\label{L3}\rho_{B}(x)\omega(y,z)+\rho_{B}(y)\omega(z,x)+\rho_{B}(z)\omega(x,y)
=\omega([x,y]_A,z)+\omega([y,z]_A,x)+\omega([z,x]_A,y),\end{equation}
\begin{equation}\label{L4}\varpi(x,\rho(y)v)-\varpi(y,\rho(x)v)-\varpi([x,y]_A,v)
=\rho_{M}(y)\varpi(x,v)-\rho_{M}(x)\varpi(y,v)-\mu(v)\omega(x,y),\end{equation}
\begin{equation}\label{L5}\rho_{M}(x)\rho_{M}(y)m-\rho_{M}(y)\rho_{M}(x)m=\rho_{M}([x,y]_A)m+\nu_{M}(\omega(x,y))m,\end{equation}
\begin{equation}\label{L6}\rho_{M}(x)\mu(v)a-\mu(\rho(x)v)a+\nu_{M}(a)\varpi(x,v)=\mu(v)\rho_{B}(x)a,\end{equation}
\begin{equation}\label{L7}\rho_{M}(x)\nu_{M}(a)m-\nu_{M}(a)\rho_{M}(x)m=\nu_{M}(\rho_B(x)a)m,\end{equation}
\begin{equation}\label{L8}\rho_{B}(Tv)S(m)+[\chi(v),S(m)]_B=S(\rho_{M}(Tv)m+\nu_{M}(\chi(v))m+\mu(v)S(m)),\end{equation}
\begin{align}\label{L9}&\rho_{B}(Tv_1)\chi(v_2)-\rho_{B}(Tv_2)\chi(v_1)+[\chi(v_1),\chi(v_2)]_B+\omega(Tv_1,Tv_2)
\\\nonumber=&S(\varpi(Tv_1,v_2)-\varpi(Tv_2,v_1))+\chi(\rho(Tv_1)v_2-\rho(Tv_2)v_1)+S(\mu(v_1)\chi(v_2)-\mu(v_2)\chi(v_1))
.\end{align}
\end{defi}

\begin{defi} Let $(\omega,\varpi,\chi,\mu,\rho_B,\rho_M)$ and
$(\omega',\varpi',\chi',\mu',\rho^{'}_{B},\rho^{'}_{M})$ be two
non-abelian 2-cocycles on $((A, [ \ , \ ]_{A}),(V,\rho),T)=\clA$ with values in
 $((B, [ \ , \ ]_{B}),(M,\nu_M),S)=\clB$. They are said to be equivalent if there
exist linear maps $\zeta:A\longrightarrow B$ and $\eta:V\longrightarrow M$
such that for all $x, y\in A, a\in B,m\in M$ and $v\in V$,
 \begin{equation}\label{E1}\omega(x,y)-\omega'(x,y)=\rho^{'}_{B}(x)\zeta(y)-\rho^{'}_{B}(y)\zeta(x)-\zeta([x,y]_A)+[\zeta(x),\zeta(y)]_B,
 \end{equation}
\begin{align}\label{E2}\rho_{B}(x)a-\rho^{'}_{B}(x)a=[\zeta(x),a]_B,
 \end{align}
\begin{equation}\label{E3}\rho_{M}(x)m-\rho^{'}_{M}(x)m=\nu_{M}(\zeta(x))m,\end{equation}
 \begin{equation}\label{E4}\mu(v)a-\mu^{'}(v)a=-\nu_{M}(a)\eta(v),\end{equation}
 \begin{equation}\label{E5}\chi(v)-\chi^{'}(v)=S\eta(v)-\zeta T(v),\end{equation}
\begin{equation}\label{E6}\varpi(x,v)-\varpi'(x,v)=\rho^{'}_{M}(x)\eta(v)+\nu_{M}(\zeta(x))\eta(v)-\mu^{'}(v)\zeta(x)-\eta(\rho(x)v).
\end{equation}
 \end{defi}

 Denote the set of all non-abelian 2-cocycles by $\mathcal{Z}^{2}_{nab}(\clA,\clB)$. The non-abelian second cohomology group
$\mathcal{H}^{2}_{nab}(\clA,\clB)$ is the quotient of $\mathcal{Z}^{2}_{nab}(\clA,\clB)$ by the above
equivalence relation.

 Let $(\omega,\varpi,\chi,\mu,\rho_B,\rho_M)$ be a non-abelian 2-cocycle
 on $((A, [ \ , \ ]_{A}),(V,\rho),T)$ with values in an abelian relative Rota-Baxter Lie algebra $((B, [ \ , \ ]_{B}),(M,\nu_M),S)$.
By Eq.~(\ref{L2}), $(B,\rho_B)$ is a representation of $A$, while, Eq.~(\ref{L5}) implies that $(M,\rho_M)$ is a representation of $A$.
 According to Eqs.~(\ref{L6}) and (\ref{L8}), we know that Eqs.~(\ref{Re1})-(\ref{Re2}) hold. Thus,
 $(\rho_B,\rho_M,\mu)$ is a representation of $((A, [ \ , \ ]_{A}),(V,\rho),T)$ on $M\xrightarrow{S}B$.
 In the light of Eqs.~(\ref{L3}), (\ref{L4}) and (\ref{L9}), we know that $(\omega,\varpi,\chi)$ is a 2-cocycle of
  $((A, [ \ , \ ]_{A}),(V,\rho),T)$ with coefficients in the representation $(M\xrightarrow{S}B,\rho_{B},
\rho_{M}, \mu)$. Thus, we have

 \begin{rmk} \label{Rk1} Let $(\omega,\varpi,\chi,\mu,\rho_B,\rho_M)$ be a non-abelian 2-cocycle
 on $((A, [ \ , \ ]_{A}),(V,\rho),T)=\clA$ with values in $((B, [ \ , \ ]_{B}),(M,\nu_M),S)=\clB$. If
 $\clB$ is abelian, then
 $(\rho_B,\rho_M,\mu)$ is a representation of $\clA$
  on $M\xrightarrow{S}B$. Furthermore, $(\omega,\varpi,\chi)$ is a 2-cocycle of
  $\clA$ with
 coefficients in the representation $(M\xrightarrow{S}B,\rho_{B},
\rho_{M}, \mu)$ and $\mathcal{H}_{nab}^{2}\clA,\clB)=\mathcal{H}^{2}(\clA,\clB)$.
 \end{rmk}

Using the above notations, we define maps $[  \ ,  \ ]_{\omega}:(A\oplus B)\otimes (A\oplus B)\longrightarrow (A\oplus B),
~\rho_{\varpi}:A\oplus B\longrightarrow \mathfrak{gl}(V\oplus M)$
and $T_{\chi}:(V\oplus M)\longrightarrow (A\oplus B)$ respectively by
\begin{align}\label{NLts0}[x+a,y+b]_{\omega}&=[x,y]_{A}+\omega(x,y)+\rho_B(x)b-\rho_B(y)a+[a,b]_{B}
,\end{align}
\begin{align}\label{NLts1}\rho_{\varpi}(x+a)(v+m)=\rho(x)v+\rho_{M}(x)m-\mu(v)a+\varpi(x,v)+\nu_{M}(a)m,\end{align}
\begin{align}\label{NLts2}T_{\chi}(v+m)=T(v)+S(m)+\chi(v),\end{align}
for all $x,y\in A,a,b\in B,v\in V$ and $m\in M$.

\begin{pro} \label{LY} With the above notions,
$((A\oplus B,[  \  ,  \ ]_{\omega}),(V\oplus M,\rho_{\varpi}),T_{\chi})$ is a relative Rota-Baxter Lie algebra
 if and only if the sextuple $(\omega,\varpi,\chi,\mu,\rho_B,\rho_M)$
is a non-abelian 2-cocycle on $((A, [ \ , \ ]_{A}),(V,\rho),T)=\clA$ 
with values in $((B, [ \ , \ ]_{B}),(M,\nu_M),S)=\clB$. 
Denote this relative Rota-Baxter Lie algebra
  simply by $\clA\oplus_{(\omega,\varpi,\chi)} \clB$.
\end{pro}

\begin{proof}
$((A\oplus B,[  \  ,  \ ]_{\omega}),(V\oplus W,\rho_{\varpi}),T_{\chi})$ is a relative Rota-Baxter Lie algebra if and only if
(i) $(A\oplus B,[  \  ,  \ ]_{\omega})$ is a Lie algebra. (ii) $(V\oplus M,\rho_{\varpi})$ is a representation of
$(A\oplus B,[  \  ,  \ ]_{\omega})$. 
(iii) $T_{\chi}$ is a relative Rota-Baxter operator associated with the representation $(V\oplus M,\rho_{\varpi})$.
Under the context of Lie algebras, $(A\oplus B,[  \  ,  \ ]_{\omega})$ is a Lie algebra if and only if Eqs.~(\ref{L1})-(\ref{L3}) hold.
 By direct computations, we can check that
$(V\oplus M,\rho_{\varpi})$ is a representation of $(A\oplus B,[  \  ,  \ ]_{\omega})$ if and only if Eqs.~(\ref{L4})-(\ref{L7}) hold,
$T_{\chi}$ is a relative Rota-Baxter operator if and only if Eqs.~(\ref{L8})-(\ref{L9}) hold.

This completes the proof.
\end{proof}

\begin{defi} Let $((A, [ \ , \ ]_{A}),(V,\rho),T)=\clA$ and $((B, [ \ , \ ]_{B}),(M,\nu_M),S)=\clB$ be two
relative Rota-Baxter Lie algebras. A non-abelian
extension of $\clA$ by $\clB$ is a
relative Rota-Baxter Lie algebra $((\hat{A}, [ \ , \ ]_{\hat{A}}),(\hat{V},\hat{\rho}),\hat{T})=\clE$,
which fits into a short exact sequence of relative Rota-Baxter Lie algebras
\begin{equation}\label{Ene1} \xymatrix{
   0 \ar[r] & M \ar[d]_{S} \ar[r]^{i} & \hat{V}\ar[d]_{\hat{T}} \ar[r]^{p} & V \ar[d]_{T} \ar[r] & 0\\
  0 \ar[r] & B \ar[r]^{i} & \hat{A} \ar[r]^{\mathfrak{p}} &A  \ar[r] & 0
 .}\end{equation}
 We always write the diagram (\ref{Ene1}) simply by
 \begin{equation*}\mathcal{E}:0\longrightarrow\clB\stackrel{(\mathfrak i,i)}{\longrightarrow}
\clE\stackrel{(\mathfrak p,p)}{\longrightarrow}\clA\longrightarrow0.\end{equation*}
When $((B, [ \ , \ ]_{B}),(M,\nu_M),S)$ is abelian, the extension $\mathcal{E}$ is called 
an abelian extension of $((A, [ \ , \ ]_{A}),(V,\rho),T)$ by $((B, [ \ , \ ]_{B}),(M,\nu_M),S)$.
Denote an extension as above simply by $((\hat{A}, [ \ , \ ]_{\hat{A}}),(\hat{V},\hat{\rho}),\hat{T})=\clE$ or $\mathcal{E}$.
A section of $(\mathfrak{p},p)$ consists of
linear maps $\mathfrak{s}:A\longrightarrow\hat{A
}$ and $s:V\longrightarrow \hat{V}$ such that $\mathfrak{p}
\mathfrak{s}=I_{A}$ and $ps = I_{V}$.
\end{defi}

\begin{defi}
Let $((\hat{A}_1, [ \ , \ ]_{\hat{A}_1}),(\hat{V}_1,\hat{\rho}_{1}),\hat{T}_1)=\clE_1$ and
$((\hat{A}_2, [ \ , \ ]_{\hat{A}_2}),(\hat{V}_2,\hat{\rho}_{2}),\hat{T}_2)=\clE_2$
be two non-abelian extensions of $((A, [ \ , \ ]_{A}),(V,\rho),T)=\clA$ by $((B, [ \ , \ ]_{B}),(M,\nu_M),S)=\clB$.
They are said to be
equivalent if there is an isomorphism $(\varphi,\phi)$ of relative Rota-Baxter Lie algebras
such that the following commutative diagram
holds: \begin{equation}\label{Ene2}\xymatrixrowsep{0.36cm} \xymatrixcolsep{0.36cm} \xymatrix{
0 \ar[rr] &  & M \ar[rr] \ar[dd] \ar@{=}[rd] & & \hat{V}_1 \ar[rr] \ar[rd]^{\phi} \ar[dd] & & V \ar[dd] \ar[rr] \ar@{=}[rd] & & 0 \\
 & 0 \ar[rr] & & M \ar[rr] \ar[dd] & & \hat{V}_2 \ar[rr] \ar[dd] & & V \ar[rr] \ar[dd] & & 0 \\
0 \ar[rr] &  & B \ar[rr] \ar@{=}[rd] & & \hat{A}_1 \ar[rr] \ar[rd]^{\varphi} & & A \ar[rr] \ar@{=}[rd] & & 0 \\
 & 0 \ar[rr] & & B \ar[rr] & & \hat{A}_2 \ar[rr] & & A \ar[rr] & & 0. \\
}\end{equation}
If there is no confusion, we always write the diagram (\ref{Ene2}) simply as follows:
 \begin{equation}\label{Ene3} \xymatrix{
  0 \ar[r] & \clB\ar@{=}[d] \ar[r]^{({\mathfrak i}_1,i_1)} & \clE_1\ar[d]_{(\varphi,\phi)} \ar[r]^{({\mathfrak p}_1,p_1)} & \clA \ar@{=}[d] \ar[r] & 0\\
 0 \ar[r] & \clB \ar[r]^{({\mathfrak i}_2,i_2)} & \clE_2 \ar[r]^{({\mathfrak p}_2,p_2)} & \clA  \ar[r] & 0
 .}\end{equation}

\end{defi}

Denote by $\mathrm{Ext}_{nab}(\clA,\clB)$ the set of all equivalent classes of the non-abelian extensions of $\clA$ by $\clB$.

Suppose that $((A, [ \ , \ ]_{A}),(V,\rho),T)=\clA$ and $((B, [ \ , \ ]_{B}),(M,\nu_M),S)=\clB$ are two
relative Rota-Baxter Lie algebras. 
 Let \begin{equation*}\mathcal{E}:0\longrightarrow\clB\stackrel{(\mathfrak i,i)}{\longrightarrow}
\clE\stackrel{(\mathfrak p,p)}{\longrightarrow}\clA\longrightarrow0\end{equation*}
be a non-abelian extension of $\clA$ by $\clB$ with a section $(\mathfrak{s},s)$ of $(\mathfrak{p},p)$, where
 $\clE=((\hat{A}, [ \ , \ ]_{\hat{A}}),(\hat{V},\hat{\rho}),\hat{T})$.
Define bilinear maps $\omega_{\mathfrak{s}}:A\otimes A\rightarrow B
,~\varpi_{s}:A\otimes V\rightarrow M$ and linear maps
$\chi_{s}:V\rightarrow B,~\mu_{s}:V\rightarrow \mathrm{Hom}(B,M),~\rho_{(B,\mathfrak{s})}:A\rightarrow \mathfrak{gl}( B),
~\rho_{(M,\mathfrak{s})}:A\rightarrow \mathfrak{gl}(M)$ respectively by
\begin{equation}\label{C1}\omega_{\mathfrak{s}}(x,y)=[\mathfrak{s}(x),\mathfrak{s}(y)]_{\hat{A}}-\mathfrak{s}[x, y]_{A},
~\varpi_{(\mathfrak{s},s)}(x,v)=\hat{\rho}(\mathfrak{s}(x))s(v)-s(\rho(x)v),\end{equation}
\begin{equation}\label{C2}\chi_{(\mathfrak{s},s)}(v)=\hat{T}(s(v))-\mathfrak{s}(T(v)),~\mu_{s}(v)a=-\hat{\rho}(a)s(v),\end{equation}
\begin{equation}\label{C3}\rho_{(B,\mathfrak{s})}(x)a=[\mathfrak{s}(x), a]_{\hat{A}},
~\rho_{(M,\mathfrak{s})}(x)m=\hat{\rho}
(\mathfrak{s}(x))m\end{equation}
for any $x,y\in A,a\in B,v\in V,m\in M$.

By direct computations, we have
\begin{pro} \label{CY} With the above notions, 
$(\omega_{\mathfrak{s}},\varpi_{s},\chi_{s},\mu_{s},\rho_{(B,\mathfrak{s})},\rho_{(M,\mathfrak{s})})$ is a
non-abelian 2-cocycle on $((A, [ \ , \ ]_{A}),(V,\rho),T)$ with values in
 $((B, [ \ , \ ]_{B}),(M,\nu_M),S)$. We call it the non-abelian 2-cocycle corresponding to the 
 extension $\mathcal{E}$ induced by $(\mathfrak{s},s)$.
 Naturally, $((A\oplus B,[  \  ,  \ ]_{\omega_{\mathfrak{s}}}),(V\oplus M,\rho_{\varpi_{(\mathfrak{s},s)}}),T_{\chi_{(\mathfrak{s},s)}})$
 is a relative Rota-Baxter Lie algebra, where $[  \  ,  \ ]_{\omega_{\mathfrak{s}}},\rho_{\varpi_{(\mathfrak{s},s)}}$
  and $T_{\chi_{(\mathfrak{s},s)}}$ are defined as Eqs.~ (\ref{NLts0})-(\ref{NLts2}).
\end{pro}

In the light of Remark \ref{Rk1} and Proposition \ref{CY}, we have

\begin{rmk} 
When $\mathcal{E}$ is an abelian extension of $\clA$ by $\clB$, $(\omega_{\mathfrak{s}},\varpi_{s},\chi_{s})$ is a
 2-cocycle of $\clA$ with coefficents in
 $\clB$. We call it the 2-cocycle corresponding to the 
 extension $\mathcal{E}$ induced by $(\mathfrak{s},s)$. 
\end{rmk}

In the following, we denote $(\omega_{\mathfrak{s}},\varpi_{s},\chi_{s},\mu_{s},\rho_{(B,\mathfrak{s})},\rho_{(M,\mathfrak{s})})$ 
and $(\omega_{\mathfrak{s}},\varpi_{s},\chi_{s})$ respectively
simply by $(\omega,\varpi,\chi,\mu,\rho_{B},\rho_{M})$ and $(\omega,\varpi,\chi)$ without ambiguity. 

 \begin{lem} \label{Le1}  Assume that $((\hat{A}, [ \ , \ ]_{\hat{A}}),(\hat{V},\hat{\rho}),\hat{T})$ is
 a non-abelian extension of $((A, [ \ , \ ]_{A}),(V,\rho),T)$ by $((B, [ \ , \ ]_{B}),(M,\nu_M),S)$, that is,
 the commutative diagram (\ref{Ene1}) holds.
 Let $(\omega^i,\varpi^i,\chi^i,\mu^i,\rho^{i}_{B},\rho^{i}_M)$ be the non-abelian 2-cocycle
 corresponding to the extension $\mathcal{E}$ induced by the section $(\mathfrak{s}_i,s_i)$~(i=1,2).
 Then the two non-abelian 2-cocycles $(\omega^1,\varpi^1,\chi^1,\mu^1,\rho^{1}_{B},\rho^{1}_M)$ 
 and $(\omega^2,\varpi^2,\chi^2,\mu^2,\rho^{2}_{B},\rho^{2}_M)$
 are equivalent, that is, the equivalent classes of non-abelian 2-cocycles corresponding to
 a non-abelian extension
induced by a section are independent on the choice of sections.
\end{lem}

\begin{proof}
 Let $(\mathfrak{s}_1,s_1)$ and $(\mathfrak{s}_2,s_2)$ be two
distinct sections of $(\mathfrak{p},p)$ and $(\omega^i,\varpi^i,\chi^i,\mu^i,\rho^{i}_{B},\rho^{i}_M)$
be the corresponding non-abelian 2-cocycle induced by the section $(\mathfrak{s}_i,s_i)$~(i=1,2). Define linear maps $\zeta:
A\longrightarrow B,~\eta:V\longrightarrow M$ respectively by $\zeta(x)=\mathfrak{s}_1(x)-\mathfrak{s}_2(x),~\eta(v)=s_1(v)-s_2(v)$. Since
$\mathfrak{p}\zeta(x)=\mathfrak{p}\mathfrak{s}_1(x)-\mathfrak{p}\mathfrak{s}_2(x)=0,~p\eta(v)=ps_{1}(v)-ps_{2}(v)=0$,
 $\zeta,\eta$ are well defined. By
Eqs.~(\ref{C1})-(\ref{C3}), we have
\begin{align*}&\varpi^{1}(x,v)=\hat{\rho}(\mathfrak{s}_1(x))s_1(v)-s_1(\rho(x)v)\\=&
\hat{\rho}(\mathfrak{s}_1(x))s_2(v)+\hat{\rho}(\mathfrak{s}_1(x))\eta(v)-s_2(\rho(x)v)
-\eta(\rho(x)v)
\\=&\hat{\rho}(\mathfrak{s}_2(x))s_2(v)+\hat{\rho}(\zeta(x))s_2(v)+\hat{\rho}(\mathfrak{s}_2(x))\eta(v)
+\hat{\rho}(\zeta(x))\eta(v)-s_2(\rho(x)v)-\eta(\rho(x)v)
\\=&\varpi^{2}(x,v)-\mu^{2}(v)\zeta(x)+\rho^{2}_{M}(x)\eta(v)+\nu_{M}(\zeta(x))\eta(v)-\eta(\rho(x)v),\end{align*}
 which implies that
 Eq.~(\ref{E6}) holds. By the same token, Eqs.~ (\ref{E1})-(\ref{E5})
 hold. This finishes the proof.
\end{proof}

According to Proposition \ref{CY}, given a non-abelian extension
  \begin{equation*}\mathcal{E}:0\longrightarrow\clB\stackrel{(\mathfrak i,i)}{\longrightarrow}
\clE\stackrel{(\mathfrak p,p)}{\longrightarrow}\clA\longrightarrow0\end{equation*}
 of $((A, [ \ , \ ]_{A}),(V,\rho),T)=\clA$ by $((B, [ \ , \ ]_{B}),(M,\nu_M),S)=\clB$ 
 with a section $(\mathfrak{s},s)$ of $(\mathfrak{p},p)$, we have a non-abelian 2-cocycle
 $(\omega_{\mathfrak{s}},\varpi_{s},\chi_{s},\mu_{s},\rho_{(B,\mathfrak{s})},\rho_{(M,\mathfrak{s})})$
  and a relative Rota-Baxter Lie algebra 
  $((A\oplus B,[  \  ,  \ ]_{\omega_{\mathfrak{s}}}),(V\oplus M,\rho_{\varpi_{(\mathfrak{s},s)}}),T_{\chi_{(\mathfrak{s},s)}})$,
  denote it simply by $\clA\oplus_{(\omega_{\mathfrak{s}},\varpi_{(\mathfrak{s},s)},\chi_{(\mathfrak{s},s)})} \clB$.
 It follows that
\begin{equation*}0\longrightarrow\clB\stackrel{i}{\longrightarrow} 
\clA\oplus_{(\omega_{\mathfrak{s}},\varpi_{(\mathfrak{s},s)},\chi_{(\mathfrak{s},s)})} \clB \stackrel{\pi}{\longrightarrow}\clA\longrightarrow0\end{equation*}
 is a non-abelian extension of $\clA$ by $\clB$. Since any element
$\hat{x}\in \hat{A},\hat{v}\in \hat{V}$ can be written as $\hat{x}=a+\mathfrak s(x),\hat{v}=m+s(v)$ with $a\in B,x\in A,m\in M,v\in V$,
define a pair of linear maps $(\varphi,\phi)$ by
\begin{align*} &\varphi:(\hat{A},[ \ , \ ]_{\hat{A}})\longrightarrow (A\oplus B,[ \ , \ ]_{\omega_s}),~~~~\varphi(\hat{x})=\varphi(a+\mathfrak s(x))=a+x,
\\&\phi:(\hat{V},\hat{\rho})\longrightarrow (V\oplus_ M,\rho_{\varpi_{(\mathfrak{s},s)}}),~~~~\phi(\hat{v})=\phi(m+ s(v))=m+v.\end{align*}
It is easy to check that $(\varphi,\phi)$ is an isomorphism of relative Rota-Baxter Lie algebras such that
the following commutative diagram holds:
 \begin{equation*} \xymatrix{
 0 \ar[r] & \clB\ar@{=}[d] \ar[r]^-{(\mathfrak{i},i)} & \clE\ar[d]_-{(\varphi,\phi)} \ar[r]^-{(\mathfrak{p},p)} & \clA \ar@{=}[d] \ar[r] & 0\\
0 \ar[r] & \clB\ar[r]^-{i} & \clA\oplus_{(\omega_{\mathfrak{s}},\varpi_{(\mathfrak{s},s)},\chi_{(\mathfrak{s},s)})} \clB \ar[r]^-{\pi} & \clA  \ar[r] & 0,}\end{equation*}
 which indicates that the non-abelian extensions $\clE $ and 
 $\clA\oplus_{(\omega_{\mathfrak{s}},\varpi_{(\mathfrak{s},s)},\chi_{(\mathfrak{s},s)})} \clB$ of $\clA $ by
$\clB$ are equivalent. On the other hand, if $(\omega,\varpi,\chi,\mu,\rho_{B},\rho_M)$
 is a non-abelian 2-cocycle on $\clA$ with values in
 $\clB$, there is a relative Rota-Baxter Lie algebra $\clA\oplus_{(\omega,\varpi,\chi)} \clB$, which yields the following
 non-abelian extension of $\clA$ by $\clB$:
 \begin{equation*}0\longrightarrow\clB\stackrel{i}{\longrightarrow}\clA\oplus_{(\omega,\varpi)} \clB
\stackrel{\pi}{\longrightarrow}\clA\longrightarrow0,\end{equation*}
where $i$ is the inclusion and $\pi$ is the projection.

In the following, we focus on the relationship between non-abelian 2-cocycles and non-abelian extensions.

\begin{pro} \label{Ceg}
 Let $((A, [ \ , \ ]_{A}),(V,\rho),T)=\clA$ and $((B, [ \ , \ ]_{B}),(M,\nu_M),S)=\clB$ 
 be two relative Rota-Baxter Lie algebras.
 Then the equivalent classes of non-abelian extensions of  $((A, [ \ , \ ]_{A}),(V,\rho),T)$
  by $((B, [ \ , \ ]_{B}),(M,\nu_M),S)$
are classified by the non-abelian second cohomology group, that is,
 $\mathrm{Ext}_{nab}(\clA,\clB)\simeq \mathcal{H}_{nab}^{2}(\clA,\clB)$.
 \end{pro}
 
\begin{proof}
 Define a linear map
 \begin{equation*}\Theta:\mathrm{Ext}_{nab}(\clA,\clB)\rightarrow H_{nab}^{2}(\clA,\clB),~\end{equation*}
where $\Theta$ assigns an equivalent class of non-abelian extensions to the class of corresponding non-abelian 2-cocycles.
First, we check that $\Theta$ is well-defined.
 Assume that $((\hat{A}_1, [ \ , \ ]_{\hat{A}_1}),(\hat{V}_1,\hat{\rho}_{1}),\hat{T}_1)$ and
$((\hat{A}_2, [ \ , \ ]_{\hat{A}_2}),(\hat{V}_2,\hat{\rho}_{2}),\hat{T}_2)$
are two non-abelian extensions of $\clA$ by $\clB$, which are equivalent 
 via the map
 $(\varphi,\phi)$, that is, the commutative diagram (\ref{Ene2}) holds. Thus,
\begin{equation}\label{Erp1}\phi\hat{\rho}_1(v)=\hat{\rho}_2(\varphi(x))\phi(v),~\hat{T}_2\varphi=\phi \hat{T}_1.\end{equation}
 Let $(\mathfrak{s}_1,s_1)$ be a
section of $(\mathfrak{p}_1,p_1)$ with $\mathfrak{s}_1:A\rightarrow \hat{A}_1,~s_1:V \rightarrow \hat{V}_1.$
Then $p_2\phi s_1=p_1s_1=I_{V},\mathfrak{p}_2\varphi \mathfrak{s}_1=\mathfrak{p}_1\mathfrak{s}_1=I_{A}$, which
follows that $(\varphi \mathfrak{s}_1,\phi s_1)$ is a section of $(\mathfrak{p}_2,p_2)$. Let
 $(\omega^1,\varpi^1,\chi^1,\mu^1,\rho^{1}_{B},\rho^{1}_M)$ and
  $(\omega^2,\varpi^2,\chi^2,\mu^2,\rho^{2}_{B},\rho^{2}_M)$ be the two non-abelian 2-cocycles induced
by the sections $(\mathfrak{s}_1,s_1),(\mathfrak{s}_2,s_2)$ respectively. 
In view of Eqs.~(\ref{C1}) and (\ref{Erp1}), we obtain
\begin{align*}\varpi^{2}(x,v)&=\hat{\rho}_2(\mathfrak{s}_2(x))s_1(v)-s_2(\rho(x)v)\\&=
\hat{\rho}_2(\varphi\mathfrak{s}_1(x))\phi s_1(v)-\phi s_1(\rho(x)v)
\\&=\phi\hat{\rho}_1(\mathfrak{s}_1(x)) s_1(v)-\phi s_1(\rho(x)v)
\\&=\phi(\hat{\rho}_1(\mathfrak{s}_1(x)) s_1(v)- s_1(\rho(x)v))
\\&=\phi\varpi^{1}(x,v)
\\&=\varpi^{1}(x,v).\end{align*}
 By the same token, we have
\begin{equation*}\omega^1(x,y)=\omega^2(x,y),\chi^1(v)=\chi^2(v),\mu^1(v)=\mu^2(v),
\rho^{1}_{B}(x)=\rho^{2}_{B}(x),\rho^{1}_{M}(x)=\rho^{2}_{M}(x).\end{equation*}
 Therefore,
$(\omega^1,\varpi^1,\chi^1,\mu^1,\rho^{1}_{B},\rho^{1}_M)=(\omega^2,\varpi^2,\chi^2,\mu^2,\rho^{2}_{B},\rho^{2}_M)$,
which means that $\Theta$ is well-defined.

 Next, we verify that $\Theta$ is injective. Indeed, suppose that
 $\Theta([\mathcal{E}_1])=[(\omega^1,\varpi^1,\chi^1,\mu^1,\rho^{1}_{B},\rho^{1}_M)]$ 
 and $\Theta([\mathcal{E}_2])=[(\omega^2,\varpi^2,\chi^2,\mu^2,\rho^{2}_{B},\rho^{2}_M)]$. If
the equivalent classes 
$[(\omega^1,\varpi^1,\chi^1,\mu^1,\rho^{1}_{B},\rho^{1}_M)]=[(\omega^2,\varpi^2,\chi^2,\mu^2,\rho^{2}_{B},\rho^{2}_M)]$, 
we get that the non-abelian 2-cocycles
 $(\omega^1,\varpi^1,\chi^1,\mu^1,\rho^{1}_{B},\rho^{1}_M)$ and
$(\omega^2,\varpi^2,\chi^2,\mu^2,\rho^{2}_{B},\rho^{2}_M)$ are equivalent
via the linear maps $\zeta:A\longrightarrow
B$ and $\eta:V\longrightarrow M$, satisfying Eqs.~~(\ref{E1})-(\ref{E6}). Define linear maps
$\varphi:(A\oplus B,[ \ , \ ]_{\omega_1})\longrightarrow (A\oplus B,[ \ , \ ]_{\omega_2})$ and 
$\phi:(V\oplus M,\rho_{\varpi_1})\longrightarrow (V\oplus M,\rho_{\varpi_2})$ respectively by
\begin{equation*}\varphi(x+a)=x+\zeta(x)+a,~~\forall~x\in A,a\in B,\end{equation*}
\begin{equation*}\phi(v+m)=v+\eta(v)+m,~~\forall~v\in V,m\in M.\end{equation*}
Similarly to the proof of the converse part of Theorem \ref{EC}, we can show that 
$(\varphi,\phi)$ is an isomorphism of relative Rota-Baxter Lie algebras. Moreover,
 the following commutative diagram holds:
\begin{equation}
\xymatrix{
 \mathcal{E}_{(\omega^{1},\varpi^{1},\chi^{1})}: 0 \ar[r] & \clB\ar@{=}[d] \ar[r]^-{i} & \clA\oplus_{(\omega^{1},\varpi^{1},\chi^{1})} \clB \ar[d]_-{(\varphi,\phi)} \ar[r]^-{\pi} & \clA \ar@{=}[d] \ar[r] & 0\\
\mathcal{E}_{(\omega_2,\varpi_2,\chi_2)}: 0 \ar[r] & \clB \ar[r]^-{i} & \clA\oplus_{(\omega^{2},\varpi^{2},\chi^{2})} \clB \ar[r]^-{\pi} & \clA  \ar[r] & 0
 .}\end{equation}
Thus $\mathcal{E}_{(\omega^{1},\varpi^{1},\chi^{1})}$ and $\mathcal{E}_{(\omega^{2},\varpi^{2},\chi^{2})}$ are equivalent
non-abelian extensions of  $\clA$ by $\clB$,
which means that $[\mathcal{E}_{(\omega^{1},\varpi^{1},\chi^{1})}]=[\mathcal{E}_{(\omega^{2},\varpi^{2},\chi^{2})}]$. Thus, $\Theta$ is injective.

 Finally, we claim that $\Theta$ is surjective.
 For any equivalent class of non-abelian 2-cocycles $[(\omega,\varpi,\chi,\mu,\rho_{B},\rho_M)]$, by Proposition \ref{LY}, there is
 a non-abelian extension of  $\clA$ by $\clB$:
   \begin{equation*}\mathcal{E}_{(\omega,\varpi,\chi)}:0\longrightarrow\clB\stackrel{i}{\longrightarrow} \clA\oplus_{(\omega,\varpi,\chi)} \clB \stackrel{\pi}{\longrightarrow}\clA\longrightarrow0.\end{equation*}
   Therefore, $\Theta([\mathcal{E}_{(\omega,\varpi,\chi)}])=[(\omega,\varpi,\chi,\mu,\rho_{B},\rho_M)]$, which follows that $\Theta$ is surjective.
   In all, $\Theta$ is bijective. This finishes the proof.
 \end{proof}
 
Analogously to proposition \ref{Ceg},

\begin{thm}[\cite{35}]
	 Abelian extensions of $((A, [ \ , \ ]_{A}),(V,\rho),T)=\clA$
 by $((B, [ \ , \ ]_{B}),(M,\nu_M),S)=\clB$ are classified
by the cohomology group $\mathcal{H}^{2}(\clA,\clB)$ of
$\clA$ with coefficients in $\clB$.
\end{thm}

\section{Inducibility of a pair of relative Rota-Baxter Lie algebra automorphisms}
In this section, we study the inducibility of pairs of relative Rota-Baxter Lie algebra
automorphisms and characterize them by equivalent conditions.

Assume that $((A, [ \ , \ ]_{A}),(V,\rho),T)=\clA$ and $((B, [ \ , \ ]_{B}),(M,\nu_M),S)=\clB$ are
two relative Rota-Baxter Lie algebras. Let 
  \begin{equation*}\mathcal{E}:0\longrightarrow\clB\stackrel{(\mathfrak i,i)}{\longrightarrow}
\clE\stackrel{(\mathfrak p,p)}{\longrightarrow}\clA\longrightarrow0\end{equation*}
be a non-abelian extension of $\clA$ by $\clB$
   with a section $(\mathfrak{s},s)$ of $(\mathfrak{p},p)$ and
 $(\omega,\varpi,\chi,\mu,\rho_{B},\rho_{M})$ be the corresponding non-abelian 2-cocycle induced by 
$(\mathfrak{s},s)$, where $\clE=((\hat{A}, [ \ , \ ]_{\hat{A}}),(\hat{V},\hat{\rho}),\hat{T})$.

 Denote
$$\mathrm{Aut}_{\clB}
(\clE)=\{\gamma=(\gamma_1,\gamma_2)\in \mathrm{Aut} (\clE)\mid \gamma(\clB)\subseteq \clB, that~~ is,
\gamma_1
(B)\subseteq B,\gamma_2(M)\subseteq M\}.$$ 
\begin{pro} \label{Idc} The map $K:\mathrm{Aut}_{\clB}
(\clE)\longrightarrow \mathrm{Aut} (\clA)\times \mathrm{Aut} (\clB)$
given by
\begin{equation}\label{Ik}K(\gamma)=(\bar{\gamma},\gamma|_{\clB})~~ with~~\bar{\gamma}=(\bar{\gamma}_1,\bar{\gamma}_2)=(\mathfrak p\gamma_1 \mathfrak s,p\gamma_2 s)
,~~\gamma|_{\clB}=(\gamma_{1}|_{B},\gamma_{2}|_{M})
\end{equation}
for all $\gamma=(\gamma_1,\gamma_2)\in \mathrm{Aut}_{\clB}(\clE)$ is well defined. Moreover, $K$ is a homomorphism of groups.
\end{pro}

\begin{proof}
 According to the case of Lie algebras \cite{028}, $\bar{\gamma}_1$ 
is independent on the choice of sections and $\bar{\gamma}_1$ is an isomorphism of the Lie algebras $A$.
 Taking the same procedure, we can prove that $\bar{\gamma}_2$
does not depend on the choice of sections and $\bar{\gamma}_2$ is bijective. 
Thus, $\bar{\gamma}$ is independent on the choice of sections.
Thanks to $p(M)=0,\mathfrak p(B)=0$, we get $p\gamma_2(M)=0,\mathfrak p\gamma_1(B)=0$.
Combining Eqs.~(\ref{Irp2}) and (\ref{C1}), for all $x,y\in A$, we have 
\begin{align*}p\gamma_2 s(\rho(x)v)&=p\gamma_2 (\hat{\rho}(\mathfrak s(x))s(v)-\varpi(x,v))
\\&=p\gamma_2 \hat{\rho}(\mathfrak s(x))s(v)
\\&=p (\hat{\rho}(\gamma_1\mathfrak s(x))\gamma_2 s(v))
\\&=\rho(\mathfrak p\gamma_1\mathfrak s(x))p\gamma_2 s(v)
,\end{align*}
and using Eqs.~(\ref{Irp1}) and (\ref{C2}), we get
\begin{align*}&Tp\gamma_2 s(v)-\mathfrak p\gamma_1 \mathfrak sT(v)
\\=&\mathfrak p \hat{T}\gamma_2 s(v)-\mathfrak p\gamma_1 \mathfrak sT(v)
\\=&\mathfrak p  (\hat{T}\gamma_2 s-\gamma_1 \mathfrak sT)(v)
\\=&\mathfrak p  (\gamma_1 \hat{T}s-\gamma_1 \mathfrak sT)(v)
\\=&\mathfrak p \gamma_1\chi(v)
\\=&0
,\end{align*}
which indicates that $T\bar{\gamma}_2=\bar{\gamma}_1 T$ and $\bar{\gamma}_2 (\rho(x)v)=\rho(\bar{\gamma}_1(x))\bar{\gamma}_2 (v)$.
Thus, $\bar{\gamma}=(\bar{\gamma}_1,\bar{\gamma}_2)$ is an isomorphism of the relative Rota-Baxter Lie algebras.
It is easy to check that  $K$ is a homomorphism of groups.
\end{proof}

Since all elements of $\hat{A}$ and $\hat{V}$ can be written as $\mathfrak s(x)+a$ and $s(v)+m$ respectively with $x\in A,v\in V,a\in B,m\in M$,
 combining $\mathrm{Ker}\mathfrak{p}\cong B$ and $\mathrm{Ker}p\cong M$, we obtain that
\begin{equation*}\mathfrak p\gamma_1 \mathfrak s \mathfrak p(s(x)+a)=\mathfrak p\gamma_1 \mathfrak s (x)=\mathfrak p\gamma_1 (\mathfrak s (x)+a),\end{equation*}
\begin{equation*} p\gamma_2 s  p(s(v)+m)= p\gamma_2  s (v)= p\gamma_2 ( s (v)+m),\end{equation*}
which means that $\bar{\gamma}(\mathfrak p,p)=(\mathfrak p,p)\gamma$. 
It is clear that $\gamma(\mathfrak i,i)=(\mathfrak i,i)\gamma|_{\clB}$. Conversely, if 
$\bar{\gamma}(\mathfrak p,p)=(\mathfrak p,p)\gamma$, then 
$\bar{\gamma}=(\mathfrak p,p)\gamma(\mathfrak s,s)$. 
Thus, $(\alpha,\beta)$ is an image of $K$ if and only if there exists $\gamma\in \mathrm{Aut}_{\clE}$
such that $\alpha (\mathfrak p,p)=(\mathfrak p,p)\gamma$ and $\gamma (\mathfrak i,i)=(\mathfrak i,i)\beta$.

A pair $(\alpha,\beta)\in \mathrm{Aut} (\clA)\times \mathrm{Aut}(\clB)$ is said to
  be inducible (extensible) if $(\alpha,\beta)$ is an image of $K$.

It is natural to ask: when a pair $(\alpha,\beta)$ is
inducible. We discuss this theme in the following.

\begin{thm} \label{EC} Assume that $((A, [ \ , \ ]_{A}),(V,\rho),T)=\clA$ and $((B, [ \ , \ ]_{B}),(M,\nu_M),S)=\clB$
 are two relative Rota-Baxter Lie algebras.
Let $0\longrightarrow\clB\stackrel{(\mathfrak i,i)}{\longrightarrow}
\clE\stackrel{(\mathfrak p,p)}{\longrightarrow}\clA
\longrightarrow0$ be a non-abelian extension of $\clA$ by
$\clB$ with a section $(\mathfrak s,s)$ of $(\mathfrak p,p)$ and 
$(\omega,\varpi,\chi,\mu,\rho_{B},\rho_M)$ be the corresponding non-abelian 2-cocycle
induced by $(\mathfrak s,s)$, where $\clE=((\hat{A}, [ \ , \ ]_{\hat{A}}),(\hat{V},\hat{\rho}),\hat{T})$.
 A pair $(\alpha,\beta)\in \mathrm{Aut}(\clA
)\times \mathrm{Aut}(\clB)$ is inducible if and only if there are
linear maps $\zeta:A\longrightarrow B$ and $\eta:V\longrightarrow M$
satisfying the following conditions:
\begin{align}\label{Iam1}\beta_1\omega(x,y)-\omega(\alpha_1(x),\alpha_1(y))=\rho_{B}(\alpha_1(x))\zeta(y)-
     \rho_{B}(\alpha_1(y))\zeta(x)-\zeta([x,y]_A)+[\zeta(x),\zeta(y)]_B,
\end{align}
\begin{equation}\label{Iam2}
     \beta_1 (\rho_{B}(x)a)-\rho_{B}(\alpha_1(x))\beta_1 (a)=[\zeta(x),\beta_1 (a)]_{B},
\end{equation}
\begin{equation}\label{Iam3}
     \beta_2 (\rho_{M}(x)m)-\rho_{M}(\alpha_1(x))\beta_2 (m)=\nu_{M}(\zeta(x))\beta_{2}(m),
\end{equation}
\begin{equation}\label{Iam4}
      \beta_2 (\mu(v)a)-\mu(\alpha_2(v))\beta_1 (a)=-\nu_{M}(\beta_1 (a))\eta(v),
\end{equation}
\begin{equation}\label{Iam5}
     \beta_2\varpi(x,v)-\varpi(\alpha_1(x),\alpha_2(v))=\nu_{M}(\zeta(x))\eta(v)-\mu(\alpha_2(v))\zeta(x)+\rho_{M}(\alpha_1(x))\eta (v)-\eta (\rho(x)v),
\end{equation}
\begin{equation}\label{Iam6}
     \beta_1 \chi(v)-\chi(\alpha_2(v))=S\eta(v)-\zeta T(v).
\end{equation}
for all $x,y\in A,a\in B,v\in V$ and $m\in M$, where $\alpha=(\alpha_1,\alpha_2)$ and $\beta=(\beta_1,\beta_2)$.
\end{thm}

\begin{proof} Assume that $(\alpha,\beta)\in \mathrm{Aut}(\clA)\times \mathrm{Aut}(\clB)$ is inducible, that is, there is an
automorphism $\gamma=(\gamma_1,\gamma_2)\in \mathrm{Aut}_{\clB}(\clE)$ such
that $\gamma_{1}|_{B}=\beta_1,\gamma_{2}|_{M}=\beta_2$ and $p\gamma_2 s=\alpha_2,\mathfrak p\gamma_1 \mathfrak s=\alpha_1$. 
Since $(\mathfrak s,s)$ is a section of $(\mathfrak p,p)$,
for all
$x\in A$ and $v\in V$,
$$\mathfrak p(\mathfrak s\alpha_1-\gamma_1 \mathfrak s)(x)=\alpha_1(x)-\alpha_1(x)=0,~~ p( s\alpha_2-\gamma_2  s)(v)=\alpha_2(v)-\alpha_2(v)=0,$$
which implies that $(\mathfrak s\alpha_1-\gamma_1 \mathfrak s)(x)\in \mathrm{ker}\mathfrak p=B$ 
and $( s\alpha_2-\gamma_2  s)(v)\in \mathrm{ker} p=V$.
So we can define linear maps $\zeta:A\longrightarrow
B$ and $\eta:V\longrightarrow M$ respectively by
\begin{equation}\label{Iam7}\zeta(x)=(\gamma_1 \mathfrak s-\mathfrak s\alpha_1)(x),~~\eta(v)=(\gamma_2 s-s\alpha_2)(v),~~\forall~x\in A,~v\in V.\end{equation}
Due to $\alpha\in \mathrm{Aut}(\clA), \gamma\in \mathrm{Aut}_{\clB}(\clE)$, we get
\begin{equation}\label{Iam9}\gamma_2\hat{\rho}(\mathfrak s(x))s(v)=\hat{\rho}(\gamma_1\mathfrak s(x))\gamma_{2}s(v),\end{equation}
\begin{equation}\label{Iam10}\alpha_2(\rho(x)v)=\rho(\alpha_1(x))\alpha_2(v),\end{equation}
\begin{equation}\label{Iam11}\beta_2(\nu_{M}(a)m)=\nu_{M}(\beta_1(a))\beta_2(m).\end{equation}
 Using Eqs.~(\ref{C1})-(\ref{C3}) and (\ref{Iam7})-(\ref{Iam10}), for $x\in A, v\in V$, we have
\begin{align*}&
  \beta_2\varpi(x,v)-\varpi(\alpha_1(x),\alpha_2(v))\\=&\gamma_2\varpi(x,v)-\varpi(\alpha_1(x),\alpha_2(v))
  \\=&\gamma_2\hat{\rho}(\mathfrak s(x))s(v)-\gamma_2 s(\rho(x)v)
  -\hat{\rho}(\mathfrak s\alpha_1(x))s\alpha_2(v)+ s(\rho(\alpha_1(x))\alpha_2(v))
  \\=&\hat{\rho}(\gamma_1\mathfrak s(x))\gamma_{2}s(v)-\gamma_2 s(\rho(x)v)
  -\hat{\rho}(\gamma_1\mathfrak s(x))s\alpha_{2}(v)+\hat{\rho}(\gamma_1\mathfrak s(x))s\alpha_{2}(v)
  -\hat{\rho}(\mathfrak s\alpha_1(x))s\alpha_2(v)+s\alpha_2(\rho(x)v)
  \\=&\hat{\rho}(\gamma_1\mathfrak s(x))\eta(v)+\hat{\rho}(\zeta(x))s\alpha_2(v)-\eta (\rho(x)v)\\=&
  \hat{\rho}(\zeta(x))\eta(v)+\hat{\rho}(\mathfrak s\alpha_1(x))\eta(v)-\mu(\alpha_2(v))\zeta(x)-\eta (\rho(x)v)
  \\=&\nu_{M}(\zeta(x))\eta(v)+\rho_{M}(\alpha_1(x))\eta (v)-\mu(\alpha_2(v))\zeta(x)-\eta (\rho(x)v),
\end{align*}
which implies that Eq.~(\ref{Iam5}) holds.  
Analogously, we can show that Eqs.~(\ref{Iam1})-(\ref{Iam4}) and (\ref{Iam6}) hold.

Conversely, suppose that $(\alpha,\beta)\in \mathrm{Aut}(\clA)\times
\mathrm{Aut}(\clB)$ and there are linear maps $\zeta:A\longrightarrow B$ and
$\eta:V\longrightarrow M$
satisfying Eqs. (\ref{Iam1})-(\ref{Iam6}). Since $(\mathfrak s,s)$ is a section of $(\mathfrak p,p)$,
all $\hat{x}\in \hat{A},\hat{v}\in \hat{V}$ can be written as
$\hat{x}=a+\mathfrak s(x),\hat{v}=m+s(v)$ for some $a\in B,x\in A,m\in M,v\in V.$
Define linear maps $\gamma_1:\hat{A}\longrightarrow
\hat{A}$ and $\gamma_2:\hat{V}\longrightarrow
\hat{V}$ respectively by
\begin{equation}\label{Iam12}\gamma_1(\hat{x})=\gamma_1(a+\mathfrak s(x))=\beta_1(a)+\zeta(x)+\mathfrak s\alpha_1(x),\end{equation}
\begin{equation}\label{Iam13}\gamma_2(\hat{v})=\gamma_1(m+ s(v))=\beta_2(m)+\eta(v)+s\alpha_2(v).\end{equation}
It is easy to check that $\gamma_1,~\gamma_2$ are bijective. According to the proof of Theorem 4.1 \cite{028}, we know that
$\gamma_1\in\mathrm{Aut}(\hat{A})$ and $\gamma_1|_{B}=\beta_1,~\mathfrak p\gamma_1 \mathfrak s=\alpha_1$. In the sequel, firstly, 
we prove that $\gamma=(\gamma_1,\gamma_2),~\hat{\rho}$ 
and $\hat{T}$ satisfying Eqs.~(\ref{Irp1})-(\ref{Irp2}).
In the light of Eqs.~(\ref{C1})-(\ref{C3}) and (\ref{Iam10})-(\ref{Iam13}), we have
\begin{align*}&
    \gamma_2\hat{\rho}(\hat{x})(\hat{v})
\\=&\gamma_2\hat{\rho}(a+\mathfrak s(x))(m+s(v))
\\=&\gamma_2(\hat{\rho}(a)m+\hat{\rho}(\mathfrak s(x))m+\hat{\rho}(a)s(v)+\hat{\rho}(\mathfrak s(x))s(v))
\\=&\gamma_2(\hat{\rho}(a)m+\hat{\rho}(\mathfrak s(x))m+\hat{\rho}(a)s(v)+\varpi(x,v)+s(\rho(x)v))
\\=&\beta_2(\nu_M(a)m+\rho_M(x)m-\mu(v)a+\varpi(x,v))+\eta(\rho(x)v)+s\alpha_2(\rho(x)v),
\end{align*}and
\begin{align*}&
  \hat{\rho}(  \gamma_1(\hat{x}))\gamma_2(\hat{v})\\=&\hat{\rho}(  \gamma_1(a+\mathfrak s(x)))\gamma_2(m+s(v))
\\=&\hat{\rho}(\beta_1(a)+\zeta(x)+\mathfrak s\alpha_1(x))(\beta_2(m)+\eta(v)+s\alpha_2(v))
\\=&\hat{\rho}(\beta_1(a)+\zeta(x))(\beta_2(m)+\eta(v))+
\hat{\rho}(\mathfrak s\alpha_1(x))(\beta_2(m)+\eta(v))
+\hat{\rho}(\beta_1(a)+\zeta(x))(s\alpha_2(v))\\&+\hat{\rho}(\mathfrak s\alpha_1(x))(s\alpha_2(v))
\\=&\nu_{M}(\beta_1(a)+\zeta(x))(\beta_2(m)+\eta(v))+\rho_{M}(\alpha_1(x))(\beta_2(m)+\eta(v))
-\mu(\alpha_2(v))(\beta_1(a)+\zeta(x))\\&+\varpi(\alpha_1(x),\alpha_2(v))+s(\rho(\alpha_1(x))\alpha_2(v))
\\=&\beta_2\nu_{M}((a)m)+\nu_{M}(\beta_1(a))\eta(v)+\nu_{M}(\zeta(x))(\beta_2(m)+\eta(v))+\rho_{M}(\alpha_1(x))(\beta_2(m)+\eta(v))
\\&-\mu(\alpha_2(v))(\beta_1(a)+\zeta(x))+\varpi(\alpha_1(x),\alpha_2(v))+s\alpha_2(\rho(x)v).
\end{align*}
Combining Eqs.~(\ref{Iam4})-(\ref{Iam5}), we obtain 
\begin{equation*}\gamma_2\hat{\rho}(\hat{x})(\hat{v})=\hat{\rho}( \gamma_1(\hat{x}))\gamma_2(\hat{v}).\end{equation*}
Analogously, $\hat{T}\gamma_2=\gamma_1\hat{T}$.
Thus, $\gamma\in\mathrm{Aut}(\clE)$.
Secondly, we check that $\gamma_2|_{B}=\beta_2,~p\gamma_2 s=\alpha_2$. Indeed, by Eq.~(\ref{Iam13}), we have
\begin{equation*}\gamma_2(m)=\gamma_2(m+s(0))=\beta_2(m),~\forall~m\in M\end{equation*}
\begin{equation*}p\gamma_2s(v)=p\gamma_2s(0+s(v))=p(\eta(v)+s\alpha_2(v))=\alpha_2(v),~\forall~v\in V.\end{equation*}
This completes the proof.
\end{proof}

Suppose that $((A, [ \ , \ ]_{A}),(V,\rho),T)=\clA$ and $((B, [ \ , \ ]_{B}),(M,\nu_M),S)=\clB$
are two relative Rota-Baxter Lie algebras.
Let $0\longrightarrow\clB\stackrel{(\mathfrak i,i)}{\longrightarrow}
\clE\stackrel{(\mathfrak p,p)}{\longrightarrow}\clA
\longrightarrow0$ be a non-abelian extension of $\clA$ by
$\clB$ with a section $(\mathfrak s,s)$ of $(\mathfrak p,p)$ and 
$(\omega,\varpi,\chi,\mu,\rho_{B},\rho_M)$ be the corresponding non-abelian 2-cocycle
induced by $(\mathfrak s,s)$, where $\clE=((\hat{A}, [ \ , \ ]_{\hat{A}}),(\hat{V},\hat{\rho}),\hat{T})$.
Define bilinear maps 
$\omega_{(\alpha,\beta)}:A\otimes A\rightarrow B,~\varpi_{(\alpha,\beta)}:A\otimes V\rightarrow M$ 
and linear maps $\rho_{(B,(\alpha,\beta))}:A\rightarrow \mathfrak{gl}(B),\rho_{(M,(\alpha,\beta))}:A\rightarrow \mathfrak{gl}(M),
~\mu_{(\alpha,\beta)}:V\rightarrow \mathrm{Hom}(B,M),~\chi_{(\alpha,\beta)}:V\rightarrow B$
 respectively by
 \begin{equation}\label{Inc1}\omega_{(\alpha,\beta)}(x,y)=\beta_1\omega(\alpha_{1}^{-1}(x),\alpha_{1}^{-1}(y)),~~
\varpi_{(\alpha,\beta)}(x,v)=\beta_2\varpi(\alpha_{1}^{-1}(x),\alpha_{2}^{-1}(v)),\end{equation}
  \begin{equation}\label{Inc2} \chi_{(\alpha,\beta)}(v)=\beta_1\chi(\alpha_{2}^{-1}(v)),
 ~~\mu_{(\alpha,\beta)}(v)a=\beta_2\mu(\alpha_{2}^{-1}(v))\beta_{1}^{-1}(a),\end{equation}
 \begin{equation}\label{Inc3}\rho_{(B,(\alpha,\beta))}(x)(a)=\beta_1(\rho(\alpha_{1}^{-1}(x))a),
~~\rho_{(M,(\alpha,\beta))}(x)(a)=\beta_2(\rho(\alpha_{1}^{-1}(x))a),\end{equation}
for all $x,y\in A,a\in B,m\in M,v\in V$ and
$(\alpha,\beta)\in \mathrm{Aut}(\clA)\times \mathrm{Aut}(\clB)$ with $\alpha=(\alpha_1,\alpha_2),\beta=(\beta_1,\beta_2)$.

\begin{pro} With the above notations,
$(\omega_{(\alpha,\beta)},\varpi_{(\alpha,\beta)},\chi_{(\alpha,\beta)},\mu_{(\alpha,\beta)},
\rho_{(B,(\alpha,\beta))},\rho_{(M,(\alpha,\beta))})$
is a non-abelian 2-cocycle on $\clA$ with values in
 $\clB$. \end{pro}

\begin{proof}
Since $\alpha\in \mathrm{Aut}(\clA)$, for any $x\in A,v\in V$, we have
 \begin{equation}\label{Inc4}\rho(\alpha_{2}^{-1}(y))\alpha_{2}^{-1}(v)=\alpha_{2}^{-1}(\rho(y)v).\end{equation}
By Eqs.~(\ref{L4})and (\ref{Inc1})-(\ref{Inc4}), for all $x,y\in A$ and $v\in V$, we get
\begin{align*}&\varpi_{(\alpha,\beta)}(x,\rho(y)v)-\varpi_{(\alpha,\beta)}(y,\rho(x)v)-\varpi_{(\alpha,\beta)}([x,y]_A,v)
\\=&\beta_2\varpi(\alpha_{1}^{-1}(x),\alpha_{2}^{-1}(\rho(y)v))
-\beta_2\varpi(\alpha_{1}^{-1}(y),\alpha_{2}^{-1}(\rho(x)v))
-\beta_2\varpi(\alpha_{1}^{-1}([x,y]_A),\alpha_{2}^{-1}(v))
\\=&\beta_2\varpi(\alpha_{1}^{-1}(x),\rho(\alpha_{2}^{-1}(y))\alpha_{2}^{-1}(v))
-\beta_2\varpi(\alpha_{1}^{-1}(y),\rho(\alpha_{2}^{-1}(x))\alpha_{2}^{-1}(v))
-\beta_2\varpi([\alpha_{1}^{-1}(x),\alpha_{1}^{-1}(y)]_A),\alpha_{2}^{-1}(v))
\\=&
\beta_2(\rho_{M}(\alpha_{1}^{-1}(y))
\varpi(\alpha_{1}^{-1}(x),\alpha_{2}^{-1}(v))
-\rho_{M}(\alpha_{1}^{-1}(x))\varpi(\alpha_{1}^{-1}(y),\alpha_{2}^{-1}(y))
-\mu(\alpha_{2}^{-1}(v))\omega(\alpha_{1}^{-1}(x),\alpha_{1}^{-1}(y)))
\\=&\rho_{(M,(\alpha,\beta))}(y)\varpi_{(\alpha,\beta)}(x,v)-\rho_{(M,(\alpha,\beta))}(x)\varpi_{(\alpha,\beta)}(y,v)
-\mu_{(\alpha,\beta)}(v)\omega_{(\alpha,\beta)}(x,y)
,\end{align*}
which implies that Eq.~(\ref{L4}) holds for
$(\omega_{(\alpha,\beta)},\varpi_{(\alpha,\beta)},\mu_{(\alpha,\beta)},\rho_{(M,(\alpha,\beta))})$. Similarly, we can check that
Eqs.~(\ref{L1})-(\ref{L3}) and (\ref{L5})-(\ref{L9}) hold. This finishes the proof.
\end{proof}

\begin{thm} \label{Eth1} Assume that $((A, [ \ , \ ]_{A}),(V,\rho),T)=\clA$ and 
$((B, [ \ , \ ]_{B}),(M,\nu_M),S)=\clB$ are two
 relative Rota-Baxter Lie algebras.
Let $0\longrightarrow\clB\stackrel{(\mathfrak i,i)}{\longrightarrow}
\clE\stackrel{(\mathfrak p,p)}{\longrightarrow}\clA
\longrightarrow0$ be a non-abelian extension of $\clA$ by
$\clB$ with a section $(\mathfrak s,s)$ of $(\mathfrak p,p)$ and 
$(\omega,\varpi,\chi,\mu,\rho_{B},\rho_M)$ be the corresponding non-abelian 2-cocycle
induced by $(\mathfrak s,s)$, where $\clE=((\hat{A}, [ \ , \ ]_{\hat{A}}),(\hat{V},\hat{\rho}),\hat{T})$.
 A pair $(\alpha,\beta)\in \mathrm{Aut}(\clA
)\times \mathrm{Aut}(\clB)$ is inducible if and only if the two
non-abelian 2-cocycles $(\omega,\varpi,\chi,\mu,\rho_{B},\rho_M)$
 and $(\omega_{(\alpha,\beta)},\varpi_{(\alpha,\beta)},\chi_{(\alpha,\beta)},\mu_{(\alpha,\beta)},
\rho_{(B,(\alpha,\beta))},\rho_{(M,(\alpha,\beta))})$ are equivalent.
\end{thm}

\begin{proof}
Assume that $(\alpha,\beta)\in \mathrm{Aut}(\clA)\times \mathrm{Aut}(\clB)$
is inducible, by Theorem ~\ref{EC}, there are linear maps
$\zeta:A\longrightarrow B$ and $\eta:V\longrightarrow M$ satisfying
Eqs.~~(\ref{Iam1})-(\ref{Iam7}). For all $x\in A,a\in
B$, there exist $x_0\in A,a_0\in B$ such that $x=\alpha_1(x_0),a=\beta_1(a_0)$. Thus,
by Eqs.~(\ref{Iam2}) and (\ref{Inc3}), we obtain
\begin{align*}&
\rho_{(B,(\alpha,\beta))}(x)(a)-\rho_{B}(x)(a)\\=&\beta_1\rho(\alpha_{1}^{-1}(x))\beta^{-1}(a)-\rho_{B}(x)(a)
\\=&\beta_1(\rho_{B}(x_0)a_0)-\rho_{B}(\alpha_1(x_0))(\beta_1(a_0))
\\=&[\zeta(x_0),\beta_1 (a_0)]_{B}
\\=&[\zeta\alpha_{1}^{-1}(x),a]_{B},
 \end{align*}
which means that Eq.~(\ref{E2}) holds. Analogously,
Eqs.~(\ref{E1}) and (\ref{E3})-(\ref{E6}) hold. Thus,
$(\omega,\varpi,\chi,\mu,\rho_{B},\rho_M)$
and $(\omega_{(\alpha,\beta)},\varpi_{(\alpha,\beta)},\chi_{(\alpha,\beta)},\mu_{(\alpha,\beta)},
\rho_{(B,(\alpha,\beta))},\rho_{(M,(\alpha,\beta))})$
are equivalent.

The converse part can be proved analogously.

\end{proof}

\begin{cor} Suppose that $((A, [ \ , \ ]_{A}),(V,\rho),T)=\clA$ and
 $((B, [ \ , \ ]_{B}),(M,\nu_M),S)=\clB$ are two relative Rota-baxter Lie algebras. Let
$\mathcal{E}:0\longrightarrow\clB\stackrel{(\mathfrak i, i)}{\longrightarrow}
\clE\stackrel{(\mathfrak p,p)}{\longrightarrow}\clA\longrightarrow0$
be an abelian extension of $\clA$
 by $\clB$ with a section $(\mathfrak s,s)$ of $(\mathfrak p,p)$ and
$(\omega,\varpi,\chi)$ be the corresponding 2-cocycle
induced by $(\mathfrak s,s)$. A pair $(\alpha,\beta)\in \mathrm{Aut}(\clA)\times \mathrm{Aut}(\clB)$ is inducible if and only if there are
linear maps $\zeta:A\longrightarrow B $ and $\eta:V\longrightarrow M $
satisfying the following conditions:
\begin{align}\label{AEE1}\beta_1\omega(x,y)-\omega(\alpha_1(x),\alpha_1(y))=\rho_{B}(\alpha_1(x))\zeta(y)-
     \rho_{B}(\alpha_1(y))\zeta(x)-\zeta([x,y]_A),
\end{align}
\begin{equation}\label{AEE2}
     \beta_2\varpi(x,v)-\varpi(\alpha_1(x),\alpha_2(v))=\rho_{M}(\alpha_1(x))\eta (v)-\mu(\alpha_2(v))\zeta(x)-\eta (\rho(x)v),
\end{equation}
\begin{equation}\label{AEE3}
     \beta_1 \chi(v)-\chi(\alpha_2(v))=S\eta(v)-\zeta T(v).
\end{equation}
\begin{equation}\label{AEE4}
     \beta_1 (\rho_{B}(x)a)=\rho_{B}(\alpha_1(x))\beta_1 (a),
\end{equation}
\begin{equation}\label{AEE5}
     \beta_2 (\rho_{M}(x)m)=\rho_{M}(\alpha_1(x))\beta_2 (m),
\end{equation}
\begin{equation}\label{AEE6}
      \beta_2 (\mu(v)a)=\mu(\alpha_2(v))\beta_1 (a),
\end{equation}
for all $x,y\in A,a\in B,v\in V$ and $m\in M$, where $\alpha=(\alpha_1,\alpha_2)$ and $\beta=(\beta_1,\beta_2)$.
\end{cor}

\begin{proof} 
It can be obtained directly by Theorem \ref{EC}.
\end{proof}

For all $(\alpha,\beta)\in \mathrm{Aut}(\clA)\times \mathrm{Aut}(\clB)$, $(\omega_{(\alpha,\beta)},\varpi_{(\alpha,\beta)},\chi_{(\alpha,\beta)})$ may not be a 2-cocycle.
 In fact,  $(\omega_{(\alpha,\beta)},\varpi_{(\alpha,\beta)},\chi_{(\alpha,\beta)})$
 is a 2-cocycle if Eqs.~(\ref{AEE4})-(\ref{AEE6}) hold.
 Thus, it is natural to introduce the space of compatible pairs of automorphisms:
\begin{align*}
		C_{(\mu,\rho_{B},\rho_{M})}=&\left\{\begin{aligned}&(\alpha,\beta)\in \mathrm{Aut}(\clA
)\times \mathrm{Aut}(\clB),\\&~~with~~\alpha=(\alpha_1,\alpha_2),~\beta=(\beta_1,\beta_2)
\end{aligned}\left|
\begin{aligned}& \beta_1 (\rho_{B}(x)a)=\rho_{B}(\alpha_1(x))\beta_1 (a),
\\&  \beta_2 (\rho_{M}(x)m)=\rho_{M}(\alpha_1(x))\beta_2 (m),\\& \beta_2 (\mu(v)a)=\mu(\alpha_2(v))\beta_1 (a),
\\& \forall~x\in A,a\in B,v\in V,m\in M
     \end{aligned}\right.\right\}.
	\end{align*}

Analogously to Theorem \ref{Eth1}, we get

\begin{thm}\label{Wm6} Assume that $((A, [ \ , \ ]_{A}),(V,\rho),T)=\clA$ and $((B, [ \ , \ ]_{B}),(M,\nu_M),S)=\clB$
are two relative Rota-Baxter Lie algebras. Let
$\mathcal{E}:0\longrightarrow\clB\stackrel{(\mathfrak i,i)}{\longrightarrow}
\clE\stackrel{(p,p))}{\longrightarrow}\clA\longrightarrow0$
be an abelian extension of $\clA$ by $\clB$ with a section $(\mathfrak s,s)$ of $(\mathfrak p,p)$ and
$(\omega,\varpi,\chi)$ be the corresponding 2-cocycle
induced by $(\mathfrak s,s)$.
 A pair $(\alpha,\beta)\in C_{(\mu,\rho_{B},\rho_{M})}$ is inducible
 if and only if  $(\omega,\varpi,\chi)$ and 
 $(\omega_{(\alpha,\beta)},\varpi_{(\alpha,\beta)},\chi_{(\alpha,\beta)})$
 are in the same cohomological class.
\end{thm}

\section{Wells exact sequences for relative Rota-Baxter Lie algebras}
In this section, we consider the Wells map associated with non-abelian extensions of relative Rota-Baxter Lie algebras.
Then we interpret the results gained in Section 5 in terms of the Wells map.

Assume that $((A, [ \ , \ ]_{A}),(V,\rho),T)=\clA$ and 
$((B, [ \ , \ ]_{B}),(M,\nu_M),S)=\clB$
are two relative Rota-Baxter Lie algebras.
Let $\mathcal{E}:0\longrightarrow\clB\stackrel{(\mathfrak i,i)}{\longrightarrow}
\clE\stackrel{(\mathfrak p,p)}{\longrightarrow}\clA
\longrightarrow0$ be a non-abelian extension of $\clA$ by
$\clB$ with a section $(\mathfrak s,s)$ of $(\mathfrak p,p)$ and 
$(\omega,\varpi,\chi,\mu,\rho_{B},\rho_M)$ be the corresponding non-abelian 2-cocycle
induced by $(\mathfrak s,s)$, where $\clE=((\hat{A}, [ \ , \ ]_{\hat{A}}),(\hat{V},\hat{\rho}),\hat{T})$. 

Define a  map $W:\mathrm{Aut}(\clA
)\times \mathrm{Aut}(\clB) \longrightarrow \mathcal{H}^{2}_{nab}(\clA,\clB)$ by
\begin{equation}\label{W1}
	W(\alpha,\beta)=[(\omega_{(\alpha,\beta)},\varpi_{(\alpha,\beta)},\chi_{(\alpha,\beta)},\mu_{(\alpha,\beta)},
\rho_{(B,(\alpha,\beta))},\rho_{(M,(\alpha,\beta))})
-(\omega,\varpi,\chi,\mu,\rho_{B},\rho_M)].
\end{equation}
The map $W$ is called the Wells map associated with $\mathcal{E}$. In the light of Lemma \ref{Le1}, one can easily check that
the Wells map $W$ does not depend on the choice of sections.

 \begin{thm} \label{Wm3} Let $((A, [ \ , \ ]_{A}),(V,\rho),T)=\clA$ and $((B, [ \ , \ ]_{B}),(M,\nu_M),S)=\clB$
be two relative Rota-Baxter Lie algebras.
Suppose that $\mathcal{E}:0\longrightarrow\clB\stackrel{(\mathfrak i,i)}{\longrightarrow}
\clE\stackrel{(\mathfrak p,p)}{\longrightarrow}\clA
\longrightarrow0$ is a non-abelian extension of $\clA$ by
$\clB$ with a section $(\mathfrak s,s)$ of $(\mathfrak p,p)$ and 
$(\omega,\varpi,\chi,\mu,\rho_{B},\rho_M)$ is the corresponding non-abelian 2-cocycle
induced by $(\mathfrak s,s)$, where $((\hat{A}, [ \ , \ ]_{\hat{A}}),(\hat{V},\hat{\rho}),\hat{T})=\clE$.
Then there is an exact sequence:
\begin{equation*}1\longrightarrow \mathrm{Aut}_{\clB}^{\clA}(\clE)\stackrel{H}{\longrightarrow} \mathrm{Aut}_{\clB}(\clE)
\stackrel{K}{\longrightarrow}\mathrm{Aut}(\clA)\times \mathrm{Aut}(\clB)\stackrel{W}{\longrightarrow} \mathcal{H}^{2}_{nab}(\clA,\clB),\end{equation*}
where $\mathrm{Aut}_{\clB}^{\clA}(\clE)=\{\gamma \in \mathrm{Aut}(\clE)| K(\gamma)=(I_{\clA},I_{\clB}) \}$
and $I_{\clA}=(I_A,I_V),I_{\clB}=(I_B,I_M)$.\end{thm}

\begin{proof} Obviously, $\mathrm{Ker} K=\mathrm{Im}H$ and $H$ is injective. We only need to prove that $\mathrm{Ker} W=\mathrm{Im}K$.
By Theorem \ref{Eth1}, for all $(\alpha,\beta)\in \mathrm{Ker} W$, we get that $(\alpha,\beta)$ is inducible
 with respect to the non-abelian extension $\mathcal{E}$, 
which induces that $(\alpha,\beta)\in \mathrm{Im}K$. On the other hand, for any $(\alpha,\beta)\in \mathrm{Im}K$, there is an isomorphism
$\gamma\in \mathrm{Aut}_{\clB}(\clE)$, such that
 Eq.~(\ref{Ik}) holds. Combining Theorem \ref{Eth1}, $(\alpha,\beta)\in \mathrm{Ker} W$, 
 thus, $\mathrm{Im}K\subseteq \mathrm{Ker} W$. In all, $\mathrm{Ker} W=\mathrm{Im}K$.
\end{proof}

Let $((A, [ \ , \ ]_{A}),(V,\rho),T)=\clA$ and $((B, [ \ , \ ]_{B}),(M,\nu_M),S)=\clB$
be two relative Rota-Baxter Lie algebras.
Suppose that $0\longrightarrow\clB\stackrel{(\mathfrak i,i)}{\longrightarrow}
\clE\stackrel{(\mathfrak p,p)}{\longrightarrow}\clA
\longrightarrow0$ is a non-abelian extension of $\clA$ by
$\clB$ with a section $(\mathfrak s,s)$ of $(\mathfrak p,p)$ and 
$(\omega,\varpi,\chi,\mu,\rho_{B},\rho_M)$ is the corresponding non-abelian 2-cocycle
induced by $(\mathfrak s,s)$, where $\clE=((\hat{A}, [ \ , \ ]_{\hat{A}}),(\hat{V},\hat{\rho}),\hat{T})$.

Denote
\begin{align}
		\mathcal{Z}_{nab}^{1}(\clA,\clB)=&\left\{(\zeta,\eta),~\zeta:A\rightarrow B,~\eta:V\rightarrow M \left|\begin{aligned}&
\rho_{B}(y)\zeta(x)+\zeta([x,y]_A)=\rho_{B}(x)\zeta(y)+[\zeta(x),\zeta(y)]_B,\\&
[\zeta(x),a]_{B}=0,~~\nu_{M}(\zeta(x))m=0,~~\nu_{M}(a)\eta(v)=0,\\&
\nu_{M}(\zeta(x))\eta(v)+\rho_{M}(x)\eta (v)=\eta (\rho(x)v)+\mu(v)\zeta(x),\\&
S\eta(v)=\zeta T(v),~~\forall~x,y\in A,a\in B,v\in V,m\in M.
     \end{aligned}\right.\right\}.\label{W5}
	\end{align}
It is easy to check that $\mathcal{Z}_{nab}^{1}(\clA,\clB)$ is an abelian group, which is called a non-abelian 1-cocycle on $\clA$ with values in $\clB$.
  \begin{pro} \label{Wm4} With the above notations, we have
\begin{enumerate}[label=$(\roman*)$,leftmargin=15pt]
  \item The linear map $\lambda:\mathrm{Ker} K\longrightarrow \mathcal{Z}_{nab}^{1}(\clA,\clB)$ defined by 
  $\lambda(\gamma)=(\lambda(\gamma_1),\lambda(\gamma_2))$ is
a homomorphism of groups, where
 \begin{equation}\label{W6}\lambda(\gamma_1)(x)=\zeta_{\gamma_1}(x)=\gamma_1 \mathfrak s(x)-\mathfrak s(x),~~
 \lambda(\gamma_2)(v)=\eta_{\gamma_2}(v)=\gamma_2 s(v)-s(v)\end{equation}
 for all $\gamma=(\gamma_1,\gamma_2)\in \mathrm{Ker} K,~x \in A,~v \in V$, where $K$ is defined in Proposition \ref{Idc}.
\item $\lambda$ is an isomorphism, that is, $\mathrm{Ker K}\simeq \mathcal{Z}_{nab}^{1}(\clA,\clB)$.
\end{enumerate}
\end{pro}
\begin{proof}
\begin{enumerate}[label=$(\roman*)$,leftmargin=15pt]
	\item
	By Eqs.~(\ref{C1})-(\ref{C3}) and (\ref{W6}), for all $x\in A,v\in V$, we have,
\begin{align*}&\nu_{M}(\zeta(x))\eta(v)+\rho_{M}(x)\eta (v)-\eta (\rho(x)v)-\mu(v)\zeta(x)
\\=&\hat{\rho}(\zeta(x))\eta(v)+\hat{\rho}(\mathfrak s(x))\eta(v)-\gamma_2 s(\rho(x)v)+s(\rho(x)v)
-\hat{\rho}(\zeta(x))s(v)
\\=&\hat{\rho}(\gamma_1 \mathfrak s(x))\eta(v)-\gamma_{2}s(\rho(x)v)+s(\rho(x)v)+\hat{\rho}(\gamma_1 \mathfrak s(x))s(v)-\hat{\rho}(\mathfrak s(x))s(v)
\\=&\hat{\rho}(\gamma_1 \mathfrak s(x))\gamma_{2}s(v)-\gamma_{2}s(\rho(x)v)-\varpi(x,v)
\\=&\gamma_{2}\hat{\rho}(\mathfrak s(x))s(v)-\gamma_{2}s(\rho(x)v)-\varpi(x,v)
\\=&\gamma_2\varpi(x,v)-\varpi(x,v)\\=&0.
\end{align*}
By the same token, we can prove that $\zeta_{\gamma_1},\eta_{\gamma_2}$ satisfy the other identities in $ \mathcal{Z}_{nab}^{1}(\clA,\clB)$.
Thus, $\lambda$ is well-defined.
For any $\gamma,\gamma'\in \mathrm{Ker} K$ and $x\in A$, suppose $\lambda(\gamma)=(\zeta_{\gamma_1},\eta_{\gamma_2})$ 
and $\lambda(\gamma')=(\zeta_{\gamma^{'}_{1}},\eta_{\gamma^{'}_2})$ .
Using Eq.~(\ref{W6}), we get
\begin{align*}\lambda(\gamma_1 \gamma^{'}_1)(x)&=\gamma_1 \gamma^{'}_{1}\mathfrak s(x)-\mathfrak s(x)
\\&=\gamma_1(\zeta_{\gamma'_1}(x)+\mathfrak s(x))-\mathfrak s(x)
\\&=\gamma_{1}\zeta_{\gamma'_1}(x)+\gamma_1 \mathfrak s(x)-\mathfrak s(x)
\\&=\zeta_{\gamma_1}(x)+\zeta_{\gamma'_1}(x)
\\&=\lambda(\gamma_1)(x)+\lambda(\gamma'_1)(x).\end{align*}
Take the same procedure, one can verify that 
\begin{equation*}\lambda(\gamma_2 \gamma^{'}_2)(v)=\lambda(\gamma_2)(v)+\lambda(\gamma'_2)(v).\end{equation*}
Thus, $\lambda(\gamma \gamma')=\lambda(\gamma)+\lambda(\gamma')$, that is, $\lambda$ is a homomorphism of groups.
\item For all $\gamma=(\gamma_1,\gamma_2)\in \mathrm{Ker}K$, we obtain that $K(\gamma_1)=(p\gamma_1 \mathfrak s,\gamma_{1}|_{B})=(I_A,I_B)$.
 If $\lambda(\gamma_1)=\zeta_{\gamma_1}=0$,
we can get $\zeta_{\gamma_1}(x)=\gamma_1 \mathfrak s(x)-\mathfrak s(x)=0$, that is, $\gamma_1=I_{\hat{A}}$. 
Similarly, $\gamma_2=I_{\hat{V}}$. Thus, $\lambda$ is injective.
Secondly, we prove that $\lambda$ is surjective. Since $(\mathfrak s,s)$ is a section of $(\mathfrak p,p)$, 
all $\hat{x}\in \hat{A},\hat{v}\in \hat{V}$ can be written as $a+\mathfrak s(x)$ and $m+ s(v)$ respectively for
 some $a\in B, x\in A, m\in M, v\in V$.
 For any $\zeta,\eta\in \mathcal{Z}_{nab}^{1}(\clA,\clB)$, define linear maps $\gamma_1:\hat{A}\rightarrow \hat{V}$ 
 and $\gamma_2:\hat{V}\rightarrow \hat{V}$ respectively by
  \begin{equation}\label{W7}\gamma_1(\hat{x})=\gamma_1(a+\mathfrak s(x))=\mathfrak s(x)+\zeta(x)+a,~\forall~\hat{x}\in \hat{A},\end{equation}
  \begin{equation}\label{W8}\gamma_2(\hat{v})=\gamma_1(m+ s(v))= s(v)+\eta(v)+m,~\forall~\hat{v}\in \hat{V}.\end{equation}
It is obviously that $(\mathfrak p\gamma_1 \mathfrak s,\gamma_{1}|_{B})=(I_A,I_B)$ and
 $( p\gamma_2 s,\gamma_{2}|_{M})=(I_V,I_M)$.
We need to verify that $\gamma=(\gamma_1,\gamma_2) $ is an automorphism of relative Rota-Baxter algebra $\clE$.
The proof follows the same argument as the converse part of Theorem \ref{EC}.
 It follows that $\gamma=(\gamma_1,\gamma_2) \in \mathrm{Ker} K$.
Thus, $\lambda$ is surjective. In all, $\lambda$ is bijective.
 So $\mathrm{Ker }K\simeq \mathcal{Z}_{nab}^{1}(\clA,\clB)$.
 \end{enumerate}
\end{proof}

\begin{thm} \label{Wm5} Suppose that $((A, [ \ , \ ]_{A}),(V,\rho),T)=\clA$ and
 $((B, [ \ , \ ]_{B}),(M,\nu_M),S)=\clB$ 
are two relative Rota-Baxter Lie algebras. Let
$\mathcal{E}:0\longrightarrow\clB\stackrel{(\mathfrak i,i)}{\longrightarrow}
\clE\stackrel{(\mathfrak p,p)}{\longrightarrow}\clA\longrightarrow0$
be a non-abelian extension of $\clA$ by $\clB$,
 where $\clE=((\hat{A}, [ \ , \ ]_{\hat{A}}),(\hat{V},\hat{\rho}),\hat{T})$. There is an exact sequence:
  \begin{equation*}0\longrightarrow \mathcal{Z}_{nab}^{1}(\clA,\clB)\stackrel{i}{\longrightarrow} \mathrm{Aut}_{\clB}(\clE)\stackrel{K}{\longrightarrow}\mathrm{Aut}(\clA)
  \times \mathrm{Aut}(\clB)\stackrel{W}{\longrightarrow} \mathcal{H}^{2}_{nab}(\clA,\clB).\end{equation*}
\end{thm}
\begin{proof}
The statement obtained by Theorem \ref{Wm3} and Proposition \ref{Wm4}.
\end{proof}

In the case of abelian extensions, $\mathcal{Z}^{1}_{nab}(\clA,\clB)$
defined by (\ref{W5}) becomes to $\mathcal{Z}^{1}(\clA,\clB)$ defined by Eq.~(\ref{Cy1}).
 In the light of Theorem \ref{Wm5} and Theorem \ref{Wm6}, we get 

\begin{thm} Let $\mathcal{E}:0\longrightarrow\clB\stackrel{(\mathfrak i,i)}{\longrightarrow} \clE
\stackrel{(\mathfrak p,p)}{\longrightarrow}\clA\longrightarrow0$ be an abelian extension of $\clA$
by $\clB$. There is an exact sequence:
$$0\longrightarrow \mathcal{Z}^{1}(\clA,\clB)\stackrel{i}{\longrightarrow} \mathrm{Aut}_{\clB}(\clE)
\stackrel{K}{\longrightarrow}C_{(\mu,\rho_{B},\rho_{M})}\stackrel{W}{\longrightarrow} \mathcal{H}^{2}(\clA,\clB).$$
\end{thm}

\section{Inducibility of pairs of derivations and Wells type exact sequences}
In this section, we address the problem of the inducibility of pairs of derivations and develop Wells type exact sequences
under the context of abelian extensions of relative Rota-Baxter Lie algebras

\begin{lem} Suppose that $((A, [ \ , \ ]_{A}),(V,\rho),T)=\clA$ is a relative Rota-Baxter Lie algebra. 
and $(M\xrightarrow{S}B,\rho_{B}, \rho_{M}, \mu)=\clB$ is a representation of $\clA$.
Let $(d_{\clA},d_{\clB})\in \mathrm{Der}(\clA)\times \mathrm{End}(\clB)$. Then 
$(d_{\clA},d_{\clB})\in \mathrm{Der}(\clA\ltimes \clB)$ if and only if 
for all $x\in A,a\in B,m \in M,v \in V$ it holds that
\begin{equation}\label{Bd1}d_{M}(\rho_{M}(x)m)=\rho_{M}(x)d_{M}(m)+\rho_{M}(d_{A}x)m,
\end{equation}
\begin{equation}\label{Bd2}d_{B}(\rho_{B}(x)a)=\rho_{B}(x)d_{B}(a)+\rho_{B}
(d_{A}x)a,
\end{equation}
\begin{equation}\label{Bd3}d_{M}(\mu(v)a)=\mu(d_{V}v)a+\mu(v)d_{B}(a),
\end{equation}
\begin{equation}\label{Bd4}d_{B}S=Sd_{M},
\end{equation}
where $d_{\clA}=(d_A,d_V)$ and $d_{\clB}=(d_B,d_M)$. 
\end{lem}

\begin{proof}
It can be verified by direct computations.
\end{proof}

Consider the space \begin{equation}\label{La}\mathfrak{g}(\clA,\clB)=\{(d_{\clA},d_{\clB})\in \mathrm{Der}(\clA)\times \mathrm{End}(\clB)|
(d_{\clA},d_{\clB}) ~~satisfies~~ (\ref{Bd1})-(\ref{Bd4})\}.
\end{equation}
It is clearly that $\mathfrak{g}(\clA,\clB)$ is a subspace of $\mathrm{Der}(\clA)\times \mathrm{End}(\clB)$. Moreover,

\begin{lem} $\mathfrak{g}(\clA,\clB)$ is a Lie subalgebra of $\mathrm{Der}(\clA\ltimes\clB)$, where the Lie bracket 
of $\mathfrak{g}(\clA,\clB)$ is given by 
\begin{equation*}[(d_{\clA},d_{\clB}),(d_{\clA}^{'},d_{\clB}^{'})]=([d_{\clA},d_{\clA}^{'}],[d_{\clB},d_{\clB}^{'}])\end{equation*}
with
\begin{equation*}[d_{\clA},d_{\clA}^{'}]=([d_A,d^{'}_A],[d_V,d^{'}_{V}]),~~
[d_{\clB},d_{\clB}^{'}]=([d_B,d^{'}_B],[d_M,d^{'}_{M}])\end{equation*}
for all $d_{\clA}=(d_A,d_V),d_{\clA}^{'}=(d^{'}_A,d^{'}_{V})\in \mathrm{Der}(\clA)$ and
 $d_{\clB}=(d_B,d_M),d_{\clB}^{'}=(d^{'}_B,d^{'}_{M})\in \mathrm{End}(\clB)$.
\end{lem}
\begin{proof}
For all $d_{\clA}=(d_A,d_V),d_{\clA}^{'}=(d^{'}_A,d^{'}_{V})\in \mathrm{Der}(\clA),~
d_{\clB}=(d_B,d_M),d_{\clB}^{'}=(d^{'}_B,d^{'}_{M})\in \mathrm{End}(\clB)$ and $a\in B,v\in V$, 
using Eq.~(\ref{Bd2}), we get
 \begin{align*}[d_M,d_{M}^{'}](\mu(v)a)=&d_Md_{M}^{'}(\mu(v)a)-d_{M}^{'}d_{M}(\mu(v)a)
\\=&d_{M}(\mu(d_{V}^{'}(v))a+\mu(v)d^{'}_{B}(a))-d_{M}^{'}(\mu(d_{V}(v))a+\mu(v)d_{B}(a))
\\=&\mu(d_{V}d_{V}^{'}(v))a+\mu(d_{V}^{'}(v))d_{B}(a)+\mu(d_{V}(v))d^{'}_{B}(a)+\mu(v)(d_{B}d_{B}^{'}(a))
\\&-\mu(d_{V}^{'}d_{V}(v))a-\mu(d_{V}(v))d_{B}^{'}(a)-\mu(d_{V}^{'}(v))d_{B}(a)-\mu(v)(d_{B}^{'}d_{B}(a))
\\=&\mu([d_{V},d_{V}^{'}]v)a+\mu(v)[d_{B},d_{B}^{'}](a).
\end{align*}
 Analogously, \begin{equation*}[d_{M},d_{M}^{'}](\rho_{M}(x)m)=\rho_{M}(x)[d_{M},d_{M}^{'}](m)+\rho_{M}([d_{A},d_{A}^{'}]x)m,
\end{equation*}
\begin{equation*}[d_{M},d_{M}^{'}](\mu(v)a)=\mu([d_{V},d_{V}^{'}]v)a+\mu(v)[d_{B},d_{B}^{'}](a),
\end{equation*}
\begin{equation*}[d_{B},d_{B}^{'}]S=S[d_{M},d_{M}^{'}].
\end{equation*}
 Thus, $\mathfrak{g}(\clA,\clB)$ is a Lie subalgebra of $\mathrm{Der}(\clA\ltimes\clB)$.
\end{proof}

Define a map $\Delta_{(d_{\clA},d_{\clB})}:\mathcal{C}^{2}(\clA,\clB)\longrightarrow \mathcal{C}^{2}(\clA,\clB)$ by
\begin{align*}&\Delta_{(d_{\clA},d_{\clB})}(f_B,f_{M},\theta)\\=&
((d_{B}f_{B}-f_{B}(d_{A}\otimes I_A)-f_{B}(I_{A}\otimes d_A)),(d_{M}f_{M}-f_{M}(d_{A}\otimes I_V)-f_{M}(I_{A}\otimes d_V)),
(d_{B}\theta-\theta d_{V}))
\end{align*}
for all $(f_B,f_{M},\theta)\in \mathcal{C}^{2}(\clA,\clB)$.

In the light of Proposition 3.1 \cite{131}, there is a linear map $\Delta:\mathfrak{g}(\clA,\clB) \longrightarrow \mathrm{gl}(\mathcal{H}^{2}(\clA,\clB))$ given by
\begin{equation*}\Delta((d_{\clA},d_{\clB}))([(f_B,f_M,\theta)])=[\Delta_{(d_{\clA},d_{\clB})}(f_B,f_M,\theta)]
\end{equation*}
for all $(d_{\clA},d_{\clB})\in \mathfrak{g}(\clA,\clB)$ and $[(f_B,f_M,\theta)]\in \mathcal{H}^{2}(\clA,\clB)$.

In the remaining part, we always suppose that
  $\mathcal{E}:0\longrightarrow\clB\stackrel{(\mathfrak i, i)}{\longrightarrow}
\clE\stackrel{(\mathfrak p,p)}{\longrightarrow}\clA\longrightarrow0$
is an abelian extension of $((A, [ \ , \ ]_{A}),(V,\rho),T)=\clA$
 by $((B, [ \ , \ ]_{B}),(M,\nu_M),S)=\clB$ with a section $(\mathfrak s,s)$ of $(\mathfrak p,p)$, where $((\hat{A}, [ \ , \ ]_{\hat{A}}),(\hat{V},\hat{\rho}),\hat{T})=\clE$. 
 Assume that $(\omega,\varpi,\chi)$ is a 2-cocycle corresponding to $\mathcal{E}$ induced by $(\mathfrak s,s)$.
 
 The map $W:\mathfrak{g}(\clA,\clB)\longrightarrow \mathcal{H}^{2}(\clA,\clB)$ 
 defined by 
 \begin{equation}\label{Wm} W((d_{\clA},d_{\clB}))=\Delta(d_{\clA},d_{\clB})([(\omega,\varpi,\chi)]),
 ~~\forall~(d_{\clA},d_{\clB})\in \mathfrak{g}(\clA,\clB)\end{equation} 
 is called the Wells map associated to $\mathcal{E}$. Analogous to the proof Lemma \ref{Le1}, $W$ is independent on
 the choice of sections of $(\mathfrak p,p)$.

 Denote
$$\mathrm{Der}_{\clB}(\clE)=\{d_{\clE}=(d_{\hat{A}},d_{\hat{V}})\in \mathrm{Der} (\clE)\mid d_{\hat{A}}
(B)\subseteq B,d_{\hat{V}}(M)=M\}.$$ 
Define a linear map $\digamma:\mathrm{Der}_{\clB}(\clE)\longrightarrow \mathrm{Der} (\clA)\times \mathrm{End} (\clB)$ by
\begin{equation}\label{Dk} \digamma(d_{\clE})=(\bar{d}_{\clE},d_{\clE}|_{\clB})~~ with~~\bar{d}_{\clE}
=(\bar{d}_{\hat{A}},\bar{d}_{\hat{V}})=(\mathfrak p d_{\hat{A}} \mathfrak s,p d_{\hat{V}} s)
,~~\gamma|_{\clB}=(d_{\hat{A}}|_{B},d_{\hat{V}}|_{M})
\end{equation}
for all $d_{\clE}=(d_{\hat{A}},d_{\hat{V}})\in \mathrm{Der}_{\clB}(\clE)$.
In view of Lemma 5.1 \cite{130}, $\mathfrak p d_{\hat{A}} \mathfrak s$ is a derivation of the Lie algebra $A$ and
$\mathfrak p d_{\hat{A}} \mathfrak s$ does not depend on the choice of $\mathfrak s$. It is easy to check that 
$p d_{\hat{V}} s$ does not depend on the choice of $s$. 
Moreover, due to $(d_{\hat{A}},d_{\hat{V}})\in \mathrm{Der}_{\clB}(\clE)$, $(\mathfrak p ,p)$ being
 a homomorphism of relative Rota-Baxter Lie algebras and Eq.~(\ref{C1}), for all $x\in A,v\in V$, we have
\begin{align*}&pd_{\hat{V}}s(\rho(x)v)\\=&pd_{\hat{V}}(\hat{\rho}(\mathfrak s(x))s(v)-\varpi(x,v))
		\\=&p\hat{\rho}(\mathfrak s(x))d_{\hat{V}}s(v)+p\hat{\rho}(d_{\hat{A}}\mathfrak s(x))s(v)
\\=&\rho(\mathfrak{p}\mathfrak s(x))pd_{\hat{V}}s(v)+\rho(\mathfrak{p}d_{\hat{A}}\mathfrak s(x))ps(v)
\\=&\rho(x)pd_{\hat{V}}s(v)+\rho(\mathfrak{p}d_{\hat{A}}\mathfrak s(x))v,
	\end{align*}
and using Eq.~(\ref{C2}) and the fact $(d_{\hat{A}},d_{\hat{V}})\in \mathrm{Der}_{\clB}(\clE)$, we get
\begin{equation*}Tpd_{\hat{V}}s(v)=\mathfrak{p}\hat{T}d_{\hat{V}}s(v)
		=\mathfrak{p}d_{\hat{A}}\hat{T}s(v)
=\mathfrak{p}d_{\hat{A}}(\chi(v)+\mathfrak{s}T(v))
=\mathfrak{p}d_{\hat{A}}\mathfrak{s}T(v).
	\end{equation*}
Thus, $\bar{d}_{\clE}=(\mathfrak p d_{\hat{A}} \mathfrak s,pd_{\hat{V}} s)\in \mathrm{Der}(\clA)$, 
that is, $\digamma$ is well-defined and $\digamma$ is independent on the choice of $(\mathfrak s,s)$.

A pair $(d_{\clA},d_{\clB})\in \mathrm{Der}(\clA)\times \mathrm{End}(\clB)$ is 
called inducible (extensible) associated to $\mathcal{E}$ if $(d_{\clA},d_{\clB})$
 is an image of $\digamma$.

\begin{thm} \label{We} 
A pair $(d_{\clA},d_{\clB})\in \mathrm{Der}(\clA)\times \mathrm{End}(\clB)$ 
is inducible associated to $\mathcal{E}$ if
and only if $(d_{\clA},d_{\clB})\in\mathfrak{g}(\clA,\clB)$ and $W((d_{\clA},d_{\clB}))=[0]$.
\end{thm}

\begin{proof} Assume that the pair  $(d_{\clA},d_{\clB})\in \mathrm{Der}(\clA)\times \mathrm{End}(\clB)$ 
is inducible associated to $\mathcal{E}$. Then there is a derivation 
$d_{\clE}=(d_{\hat{A}},d_{\hat{V}})\in\mathrm{Der}_{\clB}(\clE)$ such that 
$d_{\hat{A}}|_{B}=d_{B},~~d_{\hat{V}}|_{M}= d_{M},~~d_{A}\mathfrak p=\mathfrak pd_{\hat{A}},~~d_{V}p=pd_{\hat{V}}$.
Thanks to $d_{\clE}=(d_{\hat{A}},d_{\hat{V}})\in\mathrm{Der}_{\clB}(\clE)$, we get
\begin{equation}\label{Mp1}d_{\hat{V}}(\hat{\rho}(\hat{x})\hat{v})
=\hat{\rho}(\hat{x})d_{\hat{V}}(\hat{v})+\hat{\rho}(d_{\hat{A}}(\hat{x}))\hat{v}
,~~\forall~\hat{x}\in \hat{A},\hat{v}\in \hat{V}.\end{equation} 
\begin{equation}\label{Mp2}
\hat{T}d_{\hat{V}}=d_{\hat{A}}\hat{T}.\end{equation}
Define linear maps $\zeta:A\longrightarrow B,\eta:V\longrightarrow M$ respectively by
\begin{equation}\label{Mp3}\zeta(x)=d_{\hat{A}}\mathfrak s(x)-\mathfrak {s}d_{A}(x),~~
\eta(v)=d_{\hat{V}}s(v)-sd_{V}(v),~~x\in A,v\in V.\end{equation} 
Using Eqs.~(\ref{C3}), (\ref{Mp1}) and (\ref{Mp3}), for all $x\in A,m\in M$,
we obtain
\begin{align*}&d_{M}(\rho_{M}(x)m)-\rho_{M}(x)d_{M}(m)
\\=&d_{\hat{V}}(\hat{\rho}(\mathfrak s(x))m)-\hat{\rho}(\mathfrak s(x))d_{\hat{V}}(m)
\\=&\hat{\rho}(d_{\hat{A}}\mathfrak s(x))m
\\=&\hat{\rho}(\zeta(x)+\mathfrak s d_{A}(x))m
\\=&\nu_{M}(\zeta(x))m+\hat{\rho}(\mathfrak s d_{A}(x))m
\\=&\rho_{M}(d_{A}(x))m.
\end{align*}
By the same token, we can show that Eqs.~(\ref{Bd2})-(\ref{Bd4}) hold.
It follows that $(d_{\clA},d_{\clB})\in\mathfrak{g}(\clA,\clB)$.

In the sequel, we show that 
$W((d_{\clA},d_{\clB}))=\Delta_{(d_{\clA},d_{\clB})}([(\omega,\varpi,\chi)])=[0],$
 that is, $\Delta_{(d_{\clA},d_{\clB})}([(\omega,\varpi,\chi)])=\mathcal{D_{R}}(\zeta,\eta)$. Write
 \begin{equation*}\mathcal{D_{R}}(\zeta,\eta)
=((\delta(\zeta,\eta))_{B},(\delta(\zeta,\eta))_{M},h_{T}(\zeta,\eta))
\end{equation*}
and
\begin{align*}&\Delta_{(d_{\clA},d_{\clB})}(\omega,\varpi,\chi)\\=&
((d_{B}\omega-\omega(d_{A}\otimes I_A)-\omega(I_{A}\otimes d_A)),(d_{M}\varpi-\varpi(d_{A}\otimes I_V)-\varpi(I_{A}\otimes d_V)),
(d_{B}\chi-\chi d_{V})).
\end{align*}
By Eqs.~(\ref{D2}), (\ref{C1})-(\ref{C3}), (\ref{Mp1}), (\ref{Mp3}) and (\ref{Coy2}), we have
\begin{align*}&(d_{M}\varpi-\varpi(d_{A}\otimes I_V)-\varpi(I_{A}\otimes d_V))(x,v)\\=&
d_{M}\varpi(x,v)-\varpi(d_{A}(x), v)-\varpi(x, d_V(v))
\\=&d_{\hat{V}}(\hat{\rho}(\mathfrak s(x))s(v))-d_{\hat{V}}s(\rho(x)v)
-\hat{\rho}(\mathfrak sd_{A}(x))s(v)+s(\rho(d_{A}(x))v)
-\hat{\rho}(\mathfrak s(x))sd_{V}(v)+s(\rho(x)d_{V}(v))
\\=&d_{\hat{V}}(\hat{\rho}(\mathfrak s(x))s(v))-d_{\hat{V}}s(\rho(x)v)
-\hat{\rho}(d_{\hat{A}}\mathfrak s(x)-\zeta(x))s(v)+s(\rho(d_{A}(x))v+\rho(x)d_{V}(v))
\\&-\hat{\rho}(\mathfrak s(x))(d_{\hat{V}}s(v)-\eta(v))
\\=&-d_{\hat{V}}s(\rho(x)v)
+\hat{\rho}(\zeta(x))s(v)+sd_{V}(\rho(x)v)+\hat{\rho}(\mathfrak s(x))\eta(v)
\\=&-\eta(\rho(x)v)
+\hat{\rho}(\zeta(x))s(v)+\rho_{M}(x)\eta(v)
\\=&-\eta(\rho(x)v)
-\mu(v)\zeta(x)+\rho_{M}(x)\eta(v)\\=&
(\delta(\zeta,\eta))_{M}(x,v),
\end{align*}
which yields that 
\begin{equation}\label{Wc1}d_{M}\varpi(x,v)-\varpi(d_{A}(x), v)-\varpi(x, d_V(v))=(\delta(\zeta,\eta))_{M}(x,v)\end{equation}
Take the same procedure, we can verify that
\begin{equation}\label{Wc2}d_{B}\omega(x,y)-\omega(d_{A}(x),y)-\omega(x, d_{A}(y))=(\delta(\zeta,\eta))_{B}(x,y),\end{equation}
\begin{equation}\label{Wc3}(d_{B}\chi-\chi d_{V})(v)=h_{T}(\zeta,\eta)(v).\end{equation}
Thus, $\Delta_{(d_{\clA},d_{\clB})}([(\omega,\varpi,\chi)])=\mathcal{D_{R}}(\zeta,\eta)$.

On the other hand, assume that $(d_{\clA},d_{\clB})\in\mathfrak{g}(\clA,\clB)$ and $W((d_{\clA},d_{\clB}))=[0]$.
Then there are linear maps $\zeta:A\longrightarrow B$ and $\eta:V\longrightarrow M$ such that 
$\Delta_{(d_{\clA},d_{\clB})}([(\omega,\varpi,\chi)])=\mathcal{D_{R}}(\zeta,\eta)$, that is,
Eqs.~(\ref{Wc1})-(\ref{Wc3}) hold.
 Define a 
linear map $d_{\clE}=(d_{\hat{A}},d_{\hat{V}}):\clE\longrightarrow \clE$ by
\begin{equation}\label{Mp7}d_{\hat{A}}(\mathfrak s(x)+a)=\mathfrak s d_{A}(x)+\zeta(x)+d_{B}(a),~~
d_{\hat{V}}( s(v)+m)= s d_{V}(v)+\eta(v)+d_{M}(m)\end{equation} 
for all $x\in A,m\in M,a\in B,v\in V$. 
It is clearly that $d_{\clE}|_{\clB}=d_{\clB}$. Thanks to $\mathrm{Ker}p\simeq V$, we get
\begin{equation*}p d_{\hat{V}}(s(v)+m)=p(s d_{V}(v)+\eta(v)+d_{M}(m))
=d_{V}(v)=d_{V}p(s(v)+m).
\end{equation*} 
It follows that $p d_{\hat{V}}=d_{V}p$.
Similarly, $\mathfrak p d_{\hat{A}}=d_{A}\mathfrak p$. 
By Eqs.~(\ref{D2}), (\ref{C1})-(\ref{C3}),(\ref{Bd3}) and (\ref{Mp7}), we obtain
\begin{align*}&\hat{\rho}(\mathfrak s(x)+a))d_{\hat{V}}(s(v)+m)+\hat{\rho}(d_{A}(\mathfrak s(x)+a))(s(v)+m)
\\=&\hat{\rho}(\mathfrak s(x)+a)(s d_{V}(v)+\eta(v)+d_{M}(m))+\hat{\rho}(\mathfrak s d_{A}(x)+\zeta(x)+d_{B}(a))(s(v)+m)
\\=&\hat{\rho}(\mathfrak s(x))(s d_{V}(v))+
\hat{\rho}(\mathfrak s(x))(\eta(v)+d_{M}(m))
+\hat{\rho}(a)(s d_{V}(v))+
\hat{\rho}(a)(\eta(v)+d_{M}(m))\\&+\hat{\rho}(\mathfrak s d_{A}(x))s(v)
+\hat{\rho}(\mathfrak s d_{A}(x))m+\hat{\rho}(\zeta(x)+d_{B}(a))s(v)+\hat{\rho}(\zeta(x)+d_{B}(a))m
\\=&
\varpi(x,d_{V}v)+s(\rho(x)d_{V}(v))+\rho_{M}(x)(\eta(v)+d_{M}(m))
-\mu(d_{V}v)a+\nu_{M}(a)(\eta(v)+d_{M}(m))
\\&+\varpi(d_{A}x,v)+s(\rho(d_{A}x)v)+
\rho_{M}(d_{A}x)m-\mu(v)(\zeta(x)+d_{B}a)+\nu_{M}(\zeta(x)+d_{B}a)m
\\=&
\varpi(x,d_{V}v)+sd_{V}(\rho(x)v)+\rho_{M}(x)\eta(v)+d_{M}(\rho_{M}(x)m)
-d_{M}(\mu(v)a)+\varpi(d_{A}x,v)-\mu(v)\zeta(x)
\end{align*}
and
\begin{align*}&d_{\hat{V}}(\hat{\rho}(\mathfrak s(x)+a))(s(v)+m)
\\=&d_{\hat{V}}\hat{\rho}(\mathfrak s(x))s(v)+d_{\hat{V}}\hat{\rho}(\mathfrak s(x))m
+d_{\hat{V}}\hat{\rho}(a)s(v)+d_{\hat{V}}\hat{\rho}(a)m
\\=&d_{\hat{V}}(\varpi(x,v)+s\rho(x)v)+d_{\hat{V}}\rho_{M}(x)m-d_{\hat{V}}(\mu(v)a)
+d_{\hat{V}}(\nu_{M}(a)m)
\\=&d_{M}\varpi(x,v)+\eta(\rho(x)v)+sd_{V}(\rho(x)v)+d_{M}(\rho_{M}(x)m)-d_{M}(\mu(v)a).
\end{align*}
Using Eqs.~(\ref{Wc1}) and (\ref{Coy2}), we get 
\begin{equation*}\hat{\rho}(\mathfrak s(x)+a))d_{\hat{V}}(s(v)+m)+\hat{\rho}(d_{A}(\mathfrak s(x)+a))(s(v)+m)
=d_{\hat{V}}(\hat{\rho}(\mathfrak s(x)+a))(s(v)+m).\end{equation*}
Analogously, $d_{\hat{A}}$ is a derivation of the Lie algebra $\hat{A}$ and $\hat{T}d_{\hat{V}}=d_{\hat{A}}\hat{T}$.
Thus, $d_{\clE}=(d_{\hat{A}},d_{\hat{V}})$ is a derivation of $\clE$.
In all, $(d_{\clA},d_{\clB})\in \mathrm{Der}(\clA)\times \mathrm{End}(\clB)$ 
is inducible associated to $\mathcal{E}$.
\end{proof}

\begin{pro} \label{Cb} With the above notations, we have
\begin{enumerate}[label=$(\roman*)$,leftmargin=15pt]
	\item $\mathrm{Ker}\digamma\cong \mathcal{Z}^{1}(\clA,\clB)$ as vector spaces. 
\item The map $\digamma:\mathrm{Der}_{\clB}(\clE)\longrightarrow \mathfrak{g}(\clA,\clB)$ is a homomorphism 
of Lie algebras. 
\item $\mathrm{Im}\digamma=\mathrm{Ker}W$.
\end{enumerate}
\end{pro}

\begin{proof}
(i) For all $d_{\clE}\in\mathrm{Ker}\digamma$, we get that 
$\mathfrak p d_{\hat{A}} \mathfrak s=0,~pd_{\hat{V}} s=0$
and $d_{\hat{A}}|_{B}=0=d_{\hat{V}}|_{M}$.
It follows that $d_{\hat{A}} \mathfrak s(x)\in B,~d_{\hat{V}} s(v)\in M$ for all $x\in A,v\in V $. Moreover,
by Eqs.~(\ref{eq3.3}), (\ref{C1})-(\ref{C3}), 
for all $x\in A,v\in V$, we have
\begin{align*}&(\delta\Gamma(d_{\clE}))_{M}(x,v)\\=&-\mu(v)d_{\hat{A}}\mathfrak {s}(x)
+\rho_{M}(x)d_{\hat{V}} s(v)-d_{\hat{V}}s(\rho(x)v)
\\=&\hat{\rho}(d_{\hat{A}}\mathfrak {s}(x))s(v)+\hat{\rho}(\mathfrak {s}(x))d_{\hat{V}}s(v)-
d_{\hat{V}}(\hat{\rho}(\mathfrak {s}(x))s(v)-\varpi(x,v))
\\=&0.
\end{align*}
According to Lemma 5.2 \cite{130}, $(\delta\Gamma(d_{\clE}))_{B}(x,y)=0,~~\forall~x,y\in A$.
Therefore, $\delta\Gamma(d_{\clE})\in \mathcal{Z}^{1}(\clA,\clB)$.
Hence, we can 
define a map 
\begin{equation*}\Gamma:\mathrm{Ker}\digamma\longrightarrow \mathcal{Z}^{1}(\clA,\clB),
~~\Gamma(d_{\clE})=(d_{\hat{A}}\mathfrak s,d_{\hat{V}} s),~\forall~d_{\clE}\in \mathrm{Ker}\digamma.\end{equation*}
Since $d_{\clE}|_{\clB}=0$, $\Gamma$ does not depend on the choice of $(\mathfrak s,s)$.
In the sequel, we prove that $\Gamma$ is bijective. In fact,
if $\Gamma(d_{\clE})=0$, for all $\hat{x}=\mathfrak s(x)+a\in \hat{A},
~\hat{v}=s(v)+m\in \hat{V}$ with $x\in A,a\in B,v\in V,m\in M$, we have
$d_{\hat{A}}(\mathfrak s(x)+a)=d_{\hat{A}}\mathfrak s(x)+d_{B}(a)=0$ and
$d_{\hat{V}}( s(v)+m)=d_{\hat{V}}s(v)+d_{M}(m)=0$, which indicates that $d_{\clE}=0$. 
It follows that $\Gamma$ is injective.
For any $f=(f_B,f_M)\in\mathcal{Z}^{1}(\clA,\clB)$, define a map 
$d_{\clE}=(d_{\hat{A}},d_{\hat{V}}):\clE\longrightarrow \clE$ by
\begin{equation}\label{Dr1}d_{\hat{A}}(\mathfrak s(x)+a)=f_{B}(x),~~d_{\hat{V}}( s(v)+m)=f_{M}(v),~~\forall~\mathfrak s(x)+a\in \hat{A},~s(v)+m\in \hat{V}.\end{equation}
By Lemma 5.2 \cite{130}, we know that $d_{\hat{A}}$ is a derivation of the Lie algebra $\hat{A}$.
Using Eqs.~(\ref{C1})-(\ref{C3}) and (\ref{Dr1}), we have,
\begin{align*}&d_{\hat{V}}(\hat{\rho}(\mathfrak s(x)+a)( s(v)+m))\\=&
d_{\hat{V}}(\hat{\rho}(\mathfrak s(x)) s(v)+\hat{\rho}(\mathfrak s(x))m+\hat{\rho}(a) s(v)
\hat{\rho}(a)m)
\\=&
d_{\hat{V}}(\hat{\rho}(\varpi(x,v)+s(\rho(x)v)+\rho_{M}(x)m-\mu(v)a+\nu_{M}(a)m))
\\=&f_{M}(\rho(x)v)
\end{align*}
and
\begin{align*}&\hat{\rho}(\mathfrak s(x)+a)d_{\hat{V}}( s(v)+m)+\hat{\rho}(d_{\hat{A}}(\mathfrak s(x)+a))( s(v)+m)
\\=&\hat{\rho}(\mathfrak s(x)+a)f_{M}(v)+\hat{\rho}(f_{B}(x))( s(v)+m)
\\=&\rho_{M}(x)f_{M}(v)-\mu(v)f_{B}(x)+\nu_{M}(a)f_{M}(v)+\nu_{M}(f_{B}(x))m
\\=&\rho_{M}(x)f_{M}(v)-\mu(v)f_{B}(x).
\end{align*}
Combining (\ref{Cy1}), we obtain that
\begin{equation*}d_{\hat{V}}(\hat{\rho}(\mathfrak s(x)+a)( s(v)+m))
=\hat{\rho}(\mathfrak s(x)+a)d_{\hat{V}}( s(v)+m)+\hat{\rho}(d_{\hat{A}}(\mathfrak s(x)+a))( s(v)+m).\end{equation*}
Analogously, $\hat{T}d_{\hat{V}}=d_{\hat{A}}\hat{T}$.
Thus, $d_{\clE}=(d_{\hat{A}},d_{\hat{V}})\in \mathrm{Der}(\clE)$, that is, $\Gamma$ is surjective.
In all, $\Gamma$ is bijective.

(ii) It is clear that any $(d_{\clA},d_{\clB})\in \mathrm{Im}\digamma$ is inducible.
 By Theorem \ref{We}, $\mathrm{Im}\digamma\subseteq \mathfrak{g}(\clA,\clB)$ and $\mathrm{Im}\digamma\subseteq \mathrm{Ker}W$. 
 For all $d_{\clE}=(d_{\hat{A}},d_{\hat{V}}),d_{\clE}^{'}=(d_{\hat{A}}^{'},d_{\hat{V}}^{'})\in \mathrm{Der}_{\clB}(\clE)$, 
 since $d_{\hat{V}} s(v),d_{\hat{V}}^{'} s(v)\in \hat{V}$, 
 there are elements $v_1,v_2\in V,m_1,m_2\in M$ such taht
\begin{equation}\label{Der2}d_{\hat{V}} s(v)=s(v_1)+m_1,~~d_{\hat{V}}^{'} s(v)=s(v_2)+m_2.\end{equation}
Due to $\mathrm{Ker}p\simeq M$, we get 
\begin{equation}\label{Der3} \bar{d}_{\hat{V}}(v)=pd_{\hat{V}} s(v)=p( s(v_1)+m_1)=v_1,~~
\bar{d}_{\hat{V}}^{'}(v)=pd_{\hat{V}}^{'} s(v)=p( s(v_2)+m_2)=v_2.\end{equation}
By Eqs.~(\ref{Ik}) and (\ref{Der2})-(\ref{Der3}), we have
\begin{align*}&\digamma([d_{\hat{V}},d_{\hat{V}}^{'}])(v)=\overline{[d_{\hat{V}},d_{\hat{V}}^{'}]}(v)
\\=&p(d_{\hat{V}}d_{\hat{V}}^{'}-d_{\hat{V}}^{'}d_{\hat{V}})s(v)
\\=&pd_{\hat{V}}(s(v_2)+m_2)-pd_{\hat{V}}^{'}(s(v_1)+m_1)
\\=&pd_{\hat{V}}s(v_2)-pd_{\hat{V}}^{'}s(v_1)
\\=&pd_{\hat{V}}s pd_{\hat{V}}^{'}s(v)-pd_{\hat{V}}^{'}spd_{\hat{V}}s(v)
\\=&[pd_{\hat{V}}s,pd_{\hat{V}}^{'}s](v)
\\=&[\bar{d}_{\hat{V}},\bar{d}_{\hat{V}}^{'}](v).
\end{align*}
Thanks to Lemma 5.4 \cite{130}, $\digamma([d_{\hat{A}},d_{\hat{A}}^{'}])(x)=\overline{[d_{\hat{A}},d_{\hat{A}}^{'}]}(x)
=[\bar{d}_{\hat{A}},\bar{d}_{\hat{A}}^{'}](x)$.
Thus, $\digamma([d_{\clE},d_{\clE}^{'}])=[\digamma(d_{\clE}),\digamma(d_{\clE}^{'})]$, that is,
$\digamma$ is a homomorphism of Lie algebras.

(iii) For all $(d_{\clA},d_{\clB})\in \mathrm{Ker}W$, we obtain that
$W(d_{\clA},d_{\clB})=\Delta((d_{\clA},d_{\clB}))([(\omega,\varpi,\chi)])=[0].$
Thus, there are linear maps $\zeta:A\rightarrow B$ and $\eta:V\rightarrow M$ satisfying
$\Delta((d_{\clA},d_{\clB}))([(\omega,\varpi,\chi)]])=\mathcal{D_{R}}(\zeta,\eta)$, that is,
Eqs.~(\ref{Wc1})-(\ref{Wc3}) hold.
 Define a linear map
$d_{\clE}=(d_{\hat{A}},d_{\hat{V}}):\clE\longrightarrow \clE$ by 
\begin{equation*}d_{\hat{A}}(\mathfrak s(x)+a)=\mathfrak sd_{A}(x)+\zeta(x)+d_{B}(a),
~~d_{\hat{V}}(s(v)+m)=sd_{V}(v)+\eta(v)+d_{M}(m).\end{equation*}
Taking the same procedure of the proof of Theorem \ref{We}, we have 
$d_{\clE}=(d_{\hat{A}},d_{\hat{V}})\in\mathrm{Der}_{\clB}(\clE)$
and $\digamma(d_{\clE})=(d_{\clA},d_{\clB})$. Thus, $\mathrm{Ker}W\subseteq \mathrm{Im}\digamma$.
In conclusion, $\mathrm{Im}\digamma=\mathrm{Ker}W$.
\end{proof}

\begin{thm} Assume that $((A, [ \ , \ ]_{A}),(V,\rho),T)=\clA$ and 
 $((B, [ \ , \ ]_{B}),(M,\nu_M),S)=\clB$ are two relative Rota-Baxter Lie algebras.
Let $\mathcal{E}:0\longrightarrow\clB\stackrel{(\mathfrak i, i)}{\longrightarrow}
\clE\stackrel{(\mathfrak p,p)}{\longrightarrow}\clA\longrightarrow0$
be an abelian extension of $\clA$ by $\clB$. Then there is an exact sequence given by
 $$0\longrightarrow  \mathcal{Z}^{1}(\clA,\clB)\stackrel{i}{\longrightarrow} \mathrm{Der}_{\clB}(\clE)\stackrel{\Gamma}{\longrightarrow}\mathfrak{g}(\clA,\clB)\stackrel{W}{\longrightarrow} \mathcal{H}^{2}(\clA,\clB).$$
\end{thm}

\begin{proof}
The statement follows by Proposition \ref{Cb}. 
\end{proof}


\begin{center}{\textbf{Acknowledgments}}
\end{center}
This work was supported by the Natural Science
Foundation of Zhejiang Province of China (LY19A010001); and Science
and Technology Planning Project of Zhejiang Province
(2022C01118).

\begin{center} {\textbf{Statements and Declarations}}
\end{center}
 All datasets underlying the conclusions of the paper are available
to readers. No conflict of interest exits in the submission of this
manuscript.


\end{document}